\documentclass[11pt,a4paper]{article}

\usepackage[utf8]{inputenc}
\usepackage[T1]{fontenc}
\usepackage[a4paper,margin=1in]{geometry}
\usepackage{amsmath,amssymb,amsthm,mathrsfs,bm}
\usepackage{xcolor}
\usepackage{enumitem}
 \usepackage{hyperref}
\hypersetup{
  colorlinks = true,
  linkcolor  = blue,
  citecolor  = blue,
  urlcolor   = blue,
  bookmarksnumbered = true
}

\newtheorem{theorem}{Theorem}[section]
\newtheorem{corollary}[theorem]{Corollary}
\newtheorem{lemma}[theorem]{Lemma}
\newtheorem{proposition}[theorem]{Proposition}

\theoremstyle{definition}
\newtheorem{definition}[theorem]{Definition}
\newtheorem{remark}[theorem]{Remark}

\theoremstyle{definition}

\newcommand{\N}{\mathbb{N}}
\newcommand{\Z}{\mathbb{Z}}
\newcommand{\Q}{\mathbb{Q}}
\newcommand{\R}{\mathbb{R}}

\newcommand{\cD}{\mathcal{D}}
\newcommand{\cE}{\mathcal{E}}

\newcommand{\ba}{\mathbf{a}}

\newcommand{\bx}{\mathbf{x}}

\newcommand{\by}{\mathbf{y}}

\newcommand{\bv}{\mathbf{v}}

\DeclareMathOperator{\dist}{dist}
\DeclareMathOperator{\diam}{diam}

\newcommand{\haus}{\dim_{\mathrm H}}

\newcommand{\dbar}{\overline{d}} 

 \newcommand{\dB}{d^\ast} 

\begin{document}

\title{Fractal transference principles for subsets of $\mathbb{N}^d$ of positive density}
\author{By~{\scshape Zhuowen~Guo}, {\scshape Kangbo~Ouyang},\\
{\scshape Jiahao~Qiu}, and {\scshape Shuhao~Zhang}}
\date{}

\maketitle

\begin{abstract}
We establish a multidimensional fractal transference principle for digit-restricted sets
associated with subsets of $\mathbb{N}^d$, extending the one-dimensional framework of Nakajima--Takahasi (Adv. Math., 2025).
We develop general Hausdorff-dimension tools via the singular value potential $\phi^s(\mathbf a)$
and the multivariate Dirichlet series
$\zeta_S(\boldsymbol{\sigma})
=\sum_{\mathbf a\in S}\prod_{j=1}^d a_j^{-\sigma_j}$.
Let
$s_\ast:=\inf\{0<s\le d:\sum_{\mathbf a\in S}\phi^s(\mathbf a)<\infty\}$ and
$\Lambda_S:=\inf\{\sigma_1+\cdots+\sigma_d:\zeta_S(\boldsymbol{\sigma})<\infty\}$.
We obtain $\dim_H(\mathcal E_S)\le s_\ast$, where
$\mathcal E_S\subset((0,1)\setminus\Q)^d$ denotes the set of points whose continued-fraction digit vectors
lie in $S$ and whose coordinates escape (i.e.\ $a_n(x_j)\to\infty$ for each $j$),
and $s_\ast=\tfrac12\Lambda_S$ for uniformly
$K$--balanced $S$.
If $S\subset\mathbb{N}^d$ has positive upper density, the transference theorem
constructs a set $E_S\subset\mathcal E_{\mathbb N^d}^{\mathrm{vec}}$ with
$\dim_H E_S=d/2$; in the positive upper Banach density case we can construct
$F_S\subset\mathcal E_{\mathbb N^d}^{\mathrm{vec}}$ with $\dim_H F_S=d/2$.
In both cases the common digit set recovers the corresponding density of $S$.
On the combinatorial side, the transference principle ensures that translation-invariant
configurations forced at positive density, including multidimensional Szemer\'edi patterns,
persist inside the induced fractal digit sets.
\end{abstract}

	\tableofcontents
	

\section{Introduction}
 
In this paper we study digit-restricted continued-fraction fractals through the lens of
fractal transference principles.  Roughly speaking, we show that density phenomena in
$\N^d$ can be encoded into Hausdorff-typical digit sets for continued fractions without
loss of Hausdorff dimension.  Along the way we obtain sharp Hausdorff-dimension bounds
for the resulting digit-restricted limit sets.

Regular continued fractions give a canonical symbolic model for the Gauss map and are a
cornerstone of metric Diophantine approximation; see, e.g.,
\cite{Khinchin1964,Cassels1957,IosifescuKraaikamp2002,UrbanskiZdunik2016}.
Restricting the allowed digits in the continued-fraction expansion leads to Cantor-like
limit sets in $[0,1]$, whose size is naturally measured by Hausdorff dimension.
This circle of ideas originates in the work of Good and Hirst
\cite{Good1941,Hirst1970,Hirst1973} and has been substantially developed since then
(see, e.g., \cite{Cusick1990} and references therein).
Our results fit into this tradition, but with an additional combinatorial input coming
from density questions on digit sets.

A convenient modern approach treats digit-restricted continued-fraction sets via
conformal iterated function systems and thermodynamic/symbolic methods
\cite{MauldinUrbanski1999,MUbook,Walters1982,Sarig1999}.
This framework also highlights that, beyond Hausdorff dimension, finer invariants such
as Hausdorff measure at the critical dimension and packing dimension can be subtle,
especially for infinite alphabets.
For example, Mauldin and Urbański~\cite{MauldinUrbanski1999} characterized the packing
dimension of limit sets for general conformal iterated function systems and applied this
to continued fractions with restricted digits, relating geometric size to arithmetic
density properties of the allowed entries; see also \cite{MUbook}.
In particular, the critical Hausdorff measure may behave delicately, and packing-type
substitutes and arithmetic criteria often enter (cf.\ \cite{MauldinUrbanski1999,MUbook}).

A basic theme in this area is that arithmetic growth of the allowed digits often
controls the fractal size of the corresponding limit set.
To illustrate this in dimension one, let $B$ be an infinite increasing sequence of
positive integers $b_1<b_2<\cdots$, and let $\tau$ be the exponent of convergence of the
series $\sum_{n=1}^\infty b_n^{-s}$, namely
\[
\tau:=\inf\Bigl\{s\ge 0:\ \sum_{n=1}^\infty b_n^{-s}<\infty\Bigr\}.
\]
For the limit set
\[
\bigl\{x\in(0,1)\setminus\Q:\ a_n(x)\in B\ \text{for all }n\ge1\ \text{and}\ a_n(x)\to\infty
\ \text{as }n\to\infty\bigr\},
\]
\cite{Hirst1973} gives the upper bound $\tau/2$ for its Hausdorff dimension, and
\cite{WangWu2008a} shows this bound is attained.
Related dimension problems for limsup-type digit conditions have also been studied; see,
for example, Wang--Wu~\cite{WangWu2008b}.

Digit restrictions are also closely tied to arithmetic and applications.
Bounded-digit sets connect to rational-approximation problems such as Zaremba-type
questions and to transfer-operator methods; see, e.g.,
\cite{Hensley1992,BourgainKontorovich2014,LeeZaremba2025}.
For example, Zaremba's conjecture asks whether every positive integer can occur as the
denominator of a rational number whose continued-fraction digits are uniformly bounded by
an absolute constant, and Bourgain--Kontorovich~\cite{BourgainKontorovich2014} established
this for a set of denominators of density one, highlighting links to low-discrepancy
constructions in numerical integration and pseudorandom number generation.
Continued-fraction algorithms also play a renormalization role beyond the Gauss map:
Berth\'e--Steiner--Thuswaldner~\cite{BertheSteinerThuswaldner2023} construct symbolic
codings for toral translations using multidimensional continued-fraction algorithms
realized by sequences of substitutions, obtaining Sturmian-type low complexity and
multiscale bounded remainder sets.

These examples emphasize that digit restrictions encode arithmetic information in a
geometric way.
In this paper, however, we work with the \emph{coordinatewise} continued fraction
expansion and impose digit constraints through subsets of $\N^d$.
More generally, once different directions contract at different rates, sharp
Hausdorff-dimension bounds typically require multi-scale coverings and can depend on how
the allowed pieces are distributed among directions, even in diagonal linear models;
see McMullen~\cite{McMullen1984} and the survey of Fraser~\cite{Fraser2021}, and also
Bara\'nski~\cite{Baranski2007}.
Our goal is to combine this geometric viewpoint with density input from additive
combinatorics, via a higher-dimensional fractal transference principle.

On the additive-combinatorial side, Szemer\'edi-type density theorems show that
positive-density subsets of $\N$ and $\N^d$ contain rich translation-invariant
configurations \cite{Szemeredi1975,Furstenberg1977,FurstenbergKatznelson1978},
with polynomial extensions due to Bergelson--Leibman \cite{BergelsonLeibman1996}.
A direct bridge between these themes is provided by fractal transference principles:
Nakajima--Takahasi~\cite{Nakajima2025} established a
\emph{fractal transference principle} in dimension one: translation-invariant density
statements can be transferred into digit sets inside continued-fraction fractals without
loss of Hausdorff dimension, under an injectivity constraint on digits.
The present paper develops higher-dimensional analogues for digit vectors in $\N^d$, and
complements them with sharp upper bounds for $\dim_H(\mathcal E_S)$ in terms of arithmetic
growth data of $S\subset\N^d$.

We now fix notation and explain how digit vectors in $\N^d$ encode subsets of $(0,1)^d$.

\subsection{Coordinatewise continued fractions and digit sets}

We recall the coordinatewise continued-fraction coding and introduce the digit sets that
link subsets of $\N^d$ with digit-restricted fractals in $(0,1)^d$.

Every irrational $x\in(0,1)$ admits a unique infinite regular continued fraction expansion
$x=[0;a_1(x),a_2(x),\dots]$, where $a_n(x)\in\N$ are the partial quotients.  For
$\mathbf{x}=(x_1,\dots,x_d)\in((0,1)\setminus\Q)^d$ we use the \emph{coordinatewise}
expansion and define the digit vector at time $n$ by
\[
  \mathbf a_n(\mathbf x):=\bigl(a_n(x_1),\dots,a_n(x_d)\bigr)\in\N^d.
\]
While the coordinate dynamics are independent, digit restrictions will couple the
coordinates through a subset $S\subset\N^d$.

In our transference arguments it is convenient to work in a regime where digits are
seen only once in time, so that a point naturally generates an infinite digit set.
In dimension $d=1$ we consider
\[
  E_1:=\{x\in(0,1)\setminus\Q:\ a_m(x)\neq a_n(x)\ \text{for all }m\neq n\}.
\]
The study of $E_1$ was suggested by P.~Erd\H{o}s; see Ramharter~\cite[p.~11]{Ramharter1985}.
Ramharter~\cite{Ramharter1985} proved that $\dim_H(E_1)=\tfrac12$; see also
Jordan--Rams~\cite{JordanRams2012}.  For comparison, Good's landmark theorem
\cite[Theorem~1]{Good1941} asserts that the escaping set
\[
  \{x\in(0,1)\setminus\Q:\ a_n(x)\to\infty\ \text{as }n\to\infty\}
\]
also has Hausdorff dimension $\tfrac12$.

For $d\ge1$, it is natural to consider
\[
  \mathcal E^{\mathrm{inj}}
  :=
  E_1^d
  =
  \Bigl\{\mathbf x\in((0,1)\setminus\Q)^d:
  a_m(x_j)\neq a_n(x_j)\ \text{for all }m\neq n,\ 1\le j\le d\Bigr\}.
\]
Ramharter's theorem, Marstrand's product inequality
\cite{Marstrand1954}, and Theorem~\ref{thm:intro-universal} give $  \dim_H(\mathcal E^{\mathrm{inj}})=\frac d2$.
Indeed, the first two results give the lower bound, while
$\mathcal E^{\mathrm{inj}}\subset\mathcal E_{\N^d}$ gives the upper bound.

Similar product phenomena occur in other symbolic codings.  For example,
Cheng--Zhang~\cite[Theorem~1.4]{ChengZhang2025} study exact-approximation sets arising from
$\beta$-expansions and obtain a sharp formula for the Hausdorff dimension of their
Cartesian products when $d\ge2$.  At the level of critical Hausdorff measure, even very
classical self-similar products can be delicate: Guo--Jones~\cite{GuoJones2025} provide
improved upper and lower bounds for the critical Hausdorff measure of Cartesian products
of the middle-third Cantor set, extending what was previously known in dimension two.

For each $\mathbf x\in((0,1)\setminus\Q)^d$ we define the associated digit set
\[
  D(\mathbf x):=\{\mathbf a_n(\mathbf x):n\ge 1\}\subset\N^d,
\]
and note that $D(\mathbf x)$ has no repetitions whenever $\mathbf x\in\mathcal E^{\mathrm{inj}}$.
This provides the basic interface between combinatorics on $\N^d$ and digit-restricted
fractal subsets of $(0,1)^d$.

On $\N^d$ we use the standard F{\o}lner boxes
$Q_N:=[1,N]^d\cap\N^d$ and their translates
$Q_N(\mathbf v):=\mathbf v+Q_N$, with the convention
$Q_R:=Q_{\lfloor R\rfloor}$ for real $R\ge1$.
For $S\subset\N^d$, define its upper density and upper Banach density by $$
  \dbar(S):=\limsup_{N\to\infty}\frac{|S\cap Q_N|}{|Q_N|},
  \qquad   \dB(S):=\limsup_{N\to\infty}
  \sup_{\mathbf v\in\mathbb Z_{\ge0}^d}
  \frac{|S\cap Q_N(\mathbf v)|}{|Q_N|}.
$$
If $S$ is infinite and $A\subset S$, we also write
\[
  \dbar(A\mid S)
  :=\limsup_{N\to\infty}\frac{|A\cap Q_N|}{|S\cap Q_N|},
\]
where the denominator is nonzero for all sufficiently large $N$.

Finally, for an infinite $S\subset\N^d$ we write $\widetilde{\mathcal E}_S$ and $\mathcal E_S$
for the digit-restricted sets introduced in Subsection~\ref{subsec:hausdorff-dim-digit}.
When we need the digit vectors to be injective in time, we set
\[
  \mathcal E_S^{\mathrm{vec}}
  :=
  \Bigl\{\mathbf x\in\widetilde{\mathcal E}_S:\ 
  \mathbf a_n(\mathbf x)\neq \mathbf a_m(\mathbf x)\ (n\neq m),\ 
  a_n(x_j)\to\infty\ (1\le j\le d)\Bigr\}.
\]

\subsection{Fractal transference principles on \texorpdfstring{$\N^d$}{Nd}}

Our first main theorem is an $\N^d$-analogue of the one-dimensional transference principle
of Nakajima--Takahasi~\cite[Theorem~1.1]{Nakajima2025}. Informally, it asserts that one can
choose Hausdorff-typical points in the escaping, vector-injective regime $\mathcal E_{\N^d}^{\mathrm{vec}}$
so that their digit sets reflect the density of a prescribed subset $S\subset\N^d$.

Before stating our results, we recall the one-dimensional theorem of
Nakajima--Takahasi~\cite[Theorem~1.1]{Nakajima2025}. In dimension $d=1$, they proved that if
$S\subset\N$ has $\dbar(S)>0$ (resp.\ $\dB(S)>0$), then there exist sets $E_S$
(resp.\ $F_S$) with
\[
  E_S\subset\mathcal E_{\N}^{\mathrm{vec}}\quad(\text{resp.\ }F_S\subset\mathcal E_{\N}^{\mathrm{vec}}),
  \qquad
  \dim_H E_S=\dim_H F_S=\dim_H(\mathcal E_{\N}^{\mathrm{vec}})=\frac12,
\]
such that the following two statements hold.
\begin{enumerate}[label=\textup{(\roman*)}, leftmargin=2em]
\item 
\[
  \dbar\Bigl(\bigcup_{n\ge1}\ \bigcap_{x\in E_S}\{a_n(x)\}\cap S\Bigr)=\dbar(S)
  \qquad
  (\text{resp.\ }\dB\Bigl(\bigcup_{n\ge1}\ \bigcap_{x\in F_S}\{a_n(x)\}\cap S\Bigr)=\dB(S));
\]
\item 
\[
  \dim_H\Bigl\{x\in\mathcal E_{\N}^{\mathrm{vec}}:\ \dbar(\{a_n(x):n\ge1\}\cap S)=\dbar(S)\Bigr\}
  =
  \dim_H(\mathcal E_{\N}^{\mathrm{vec}})
\]
(resp.\ with $\dB$ in place of $\dbar$).
\end{enumerate}

We now state an $\N^d$-analogue of this fractal transference principle.

\begin{theorem}[Fractal transference principle on $\N^d$]\label{thm:Zdd-fractal-transference}
Let $S\subset\N^d$.
\begin{itemize}[leftmargin=2em]
  \item[(a)] If $\dbar(S)>0$, then there exists a set
  $E_S\subset \mathcal E_{\N^d}^{\mathrm{vec}}$ such that
  \[
    \dim_H E_S=\dim_H(\mathcal E_{\N^d}^{\mathrm{vec}})=\frac d2
  \]
  and
  \begin{equation}\label{eq:ftp-a-uniform}
    \dbar\Bigl(
      \bigcup_{n\ge1}\ \bigcap_{\mathbf x\in E_S}\{\mathbf a_n(\mathbf x)\}\cap S
    \Bigr)
    =\dbar(S).
  \end{equation}
  In particular,
  \begin{equation}\label{eq:ftp-a-pointwise}
    \dim_H\Bigl\{\mathbf x\in\mathcal E_{\N^d}^{\mathrm{vec}}:\ 
    \dbar\bigl(\{\mathbf a_n(\mathbf x):n\ge1\}\cap S\bigr)=\dbar(S)\Bigr\}
    =\frac d2.
  \end{equation}

  \item[(b)] If $\dB(S)>0$, then there exists a set
  $F_S\subset \mathcal E_{\N^d}^{\mathrm{vec}}$ such that
  \[
    \dim_H F_S=\dim_H(\mathcal E_{\N^d}^{\mathrm{vec}})=\frac d2
  \]
  and
  \begin{equation}\label{eq:ftp-b-uniform}
    \dB\Bigl(
      \bigcup_{n\ge1}\ \bigcap_{\mathbf x\in F_S}\{\mathbf a_n(\mathbf x)\}\cap S
    \Bigr)
    =\dB(S).
  \end{equation}
  In particular,
  \begin{equation}\label{eq:ftp-b-pointwise}
    \dim_H\Bigl\{\mathbf x\in\mathcal E_{\N^d}^{\mathrm{vec}}:\ 
    \dB\bigl(\{\mathbf a_n(\mathbf x):n\ge1\}\cap S\bigr)=\dB(S)\Bigr\}
    =\frac d2.
  \end{equation}
\end{itemize}
\end{theorem}

\begin{remark}\label{rem:ftp-vs-NT}
When $d=1$, \eqref{eq:ftp-a-uniform}--\eqref{eq:ftp-b-pointwise} recover the
fractal transference principle of Nakajima--Takahasi~\cite[Theorem~1.1]{Nakajima2025}.
Our proof follows the same surgery--and--elimination paradigm, but several points are
genuinely high--dimensional: one must enforce coordinatewise divergence
($\min_{1\le i\le d}a_n(x_i)\to\infty$) and vector--injectivity
($\mathbf a_n(\mathbf x)\neq \mathbf a_m(\mathbf x)$ for $n\neq m$), and the dimension
analysis of the model families is anisotropic, governed by $d$--dimensional cylinder
geometry rather than one--dimensional lengths.
\end{remark}

Consequently, any translation--invariant configuration property that holds in every
subset of $\N^d$ of positive (Banach) density transfers to
$\{\mathbf a_n(\mathbf x):n\ge1\}\cap S$ for $\dim_H$--typical
$\mathbf x\in\mathcal E_{\N^d}^{\mathrm{vec}}$.
For instance, combining Theorem~\ref{thm:Zdd-fractal-transference} with the polynomial
Szemer\'edi theorem of Bergelson--Leibman~\cite{BergelsonLeibman1996} yields the following.

 \begin{corollary}[Polynomial configurations in digit sets]\label{cor:ftp-polynomial}
Fix $r\in\N$. Let $P:\Z^r\to\Z^d$ be a polynomial mapping with $P(\mathbf 0)=\mathbf 0$,
and let $F\subset\Z^r$ be finite. Let $S\subset\N^d$.

\begin{itemize}[leftmargin=2em]
  \item[(a)] If $\dbar(S)>0$, then there exists $E_S\subset\mathcal E_{\N^d}^{\mathrm{vec}}$ with
  $\dim_H E_S=d/2$ such that the set
  \[
    \bigcup_{n\ge1}\ \bigcap_{\mathbf x\in E_S}\{\mathbf a_n(\mathbf x)\}\cap S
  \]
  has positive upper density in $\N^d$. Consequently, applying the polynomial
  Szemer\'edi theorem of Bergelson--Leibman~\cite{BergelsonLeibman1996} to
  $F\cup\{\mathbf0\}$, so that the base point itself belongs to the displayed set,
  there exist $n\in\N$ and $\mathbf u\in\N^d$ such that for every
  $\mathbf x\in E_S$ one has
  \[
    \mathbf u+P(nF)\subset \{\mathbf a_m(\mathbf x):m\ge1\}\cap S.
  \]

  \item[(b)] If $\dB(S)>0$, then there exists $F_S\subset\mathcal E_{\N^d}^{\mathrm{vec}}$ with
  $\dim_H F_S=d/2$ such that the set
  \[
    \bigcup_{n\ge1}\ \bigcap_{\mathbf x\in F_S}\{\mathbf a_n(\mathbf x)\}\cap S
  \]
  has positive upper Banach density in $\N^d$. Consequently, applying
  Bergelson--Leibman~\cite{BergelsonLeibman1996} to $F\cup\{\mathbf0\}$,
  so that the base point itself belongs to the displayed set, there exist $n\in\N$
  and $\mathbf u\in\N^d$ such that for every $\mathbf x\in F_S$ one has
  \[
    \mathbf u+P(nF)\subset \{\mathbf a_m(\mathbf x):m\ge1\}\cap S.
  \]
\end{itemize}
\end{corollary}

\subsection{Relative fractal transference on \texorpdfstring{$\N^d$}{Nd}}

We also develop a relative version of Nakajima--Takahasi~\cite[Theorem~1.4]{Nakajima2025},
where a sparse ambient set $S\subset\N^d$ plays the role of a background set.
 For sharp upper bounds it is convenient to
control anisotropy: for $K\ge 1$ let
\[
  \mathcal C_K:=\Bigl\{\mathbf a\in\N^d:\ \frac{\max_j a_j}{\min_j a_j}\le K\Bigr\},
\]
and call $S$ \emph{uniformly $K$-balanced} if $S\subset\mathcal C_K$.

To quantify sparsity, we say that $S$ has \emph{polynomial density of exponent $\alpha\ge 1$} if
\[
  |S\cap Q_N|\asymp \frac{N^{d/\alpha}}{(\log N)^\beta}
  \qquad\text{for some }\beta\ge0,
\]
where for two nonnegative quantities depending on $N$ we write $A_N\asymp B_N$ if there exist constants $c,C>0$ such that
$c\,B_N\le A_N\le C\,B_N$ for all sufficiently large $N$.

Before stating our higher-dimensional result, we recall the one-dimensional theorem of Nakajima--Takahasi~\cite[Theorem~1.4]{Nakajima2025}.     They proved that if $S\subset\N$ is infinite
with polynomial density of exponent $\alpha\ge1$ and $A\subset S$ satisfies
$\dbar(A\mid S)>0$, then there exists
$E_{S,A}\subset \mathcal E_{\N}^{\mathrm{vec}}$ with
\[
  \{a_n(x):n\ge1\}\subset S\ \text{ for all }x\in E_{S,A},
  \qquad
  \dim_H E_{S,A}
  =
  \dim_H\bigl(\mathcal E_S^{\mathrm{vec}}\bigr)
  =
  \frac{1}{2\alpha},
\]
such that the following two statements hold.
\begin{enumerate}[label=\textup{(\roman*)}, leftmargin=2em]
\item
\[
  \dbar\Bigl(
    \bigcup_{n\ge1}\ \bigcap_{x\in E_{S,A}}\{a_n(x)\}\cap A
    \ \bigm|\ S
  \Bigr)
  =
  \dbar(A\mid S);
\]
\item
Writing $D(x):=\{a_n(x):n\ge1\}$, one has
\[
\dim_H\Bigl\{x\in \mathcal E_{\N}^{\mathrm{vec}}:
D(x)\subset S,\ 
\dbar(D(x)\cap A\mid S)=\dbar(A\mid S)\Bigr\}
=
\dim_H\bigl(\mathcal E_S^{\mathrm{vec}}\bigr).
\]
\end{enumerate}

We now state an $\N^d$-analogue of this relative fractal transference principle.

\begin{theorem}[Relative fractal transference on $\N^d$]\label{thm:Zdd-relative}
Let $S\subset\mathcal C_K\subset\N^d$ be infinite with polynomial density of exponent
$\alpha\ge1$.  For any $A\subset S$ with $\dbar(A\mid S)>0$, there exists a set
$E_{S,A}\subset \mathcal E_{S}^{\mathrm{vec}}$ such that
\[
  \dim_H E_{S,A}
  =
  \dim_H\bigl(\mathcal E_S^{\mathrm{vec}}\bigr)
  =
  \frac{d}{2\alpha},
\]
and
\[
  \dbar\Bigl(
    \bigcup_{n\ge1}\ \bigcap_{\bx\in E_{S,A}}\{\ba_n(\bx)\}\cap A
    \ \Bigm|\ S
  \Bigr)
  =
  \dbar(A\mid S),
\]
and moreover
\[
  \dim_H\Bigl\{
    \bx\in \mathcal E_S^{\mathrm{vec}}:\ 
    \dbar\bigl(D(\bx)\cap A\mid S\bigr)=\dbar(A\mid S)
  \Bigr\}
  =
  \dim_H\bigl(\mathcal E_S^{\mathrm{vec}}\bigr).
\]
\end{theorem}

\begin{remark}
When $d=1$ the balance condition is automatic, so Theorem~\ref{thm:Zdd-relative}
recovers Nakajima--Takahasi~\cite[Theorem~1.4]{Nakajima2025}.
\end{remark}

In higher dimensions, the balance condition is needed because different
coordinates may contract at different rates, and then a single counting
exponent may not determine the Hausdorff dimension.  This phenomenon already happens in
diagonal self-affine constructions, where the dimension depends on how the
allowed symbols are distributed; see
\cite{McMullen1984,LalleyGatzouras1992,Falconer1988,HueterLalley1995,
Baranski2007,Fraser2021}.  Related overlap effects in product settings are
studied in \cite{PeresShmerkin2009,HochmanShmerkin2012}; for
transfer-operator methods in continued-fraction dimension problems, see
\cite{JenkinsonPollicott2001,Hensley2012}.

\subsection{Hausdorff dimension of digit-restricted sets}
\label{subsec:hausdorff-dim-digit}
Finally, we record our dimension estimates for digit-restricted limit sets. They give a universal upper bound and, under mild arithmetic hypotheses on $S\subset\N^d$, sharp bounds expressed in terms of a singular-value potential.

Let $S\subset\N^d$ and define
\begin{equation}\label{eq:ES-tilde-def}
  \widetilde{\mathcal E}_S
  :=
  \{\mathbf x\in((0,1)\setminus\Q)^d:\ \mathbf a_n(\mathbf x)\in S\ \forall n\ge 1\},
  \qquad
  \mathcal E_S
  :=
  \{\mathbf x\in\widetilde{\mathcal E}_S:\ a_n(x_j)\to\infty\ (1\le j\le d)\}.
\end{equation}
We also recall the vector-injective subregime
\[
  \mathcal E_S^{\mathrm{vec}}
  :=
  \Bigl\{\mathbf x\in\widetilde{\mathcal E}_S:\ 
  \mathbf a_n(\mathbf x)\neq \mathbf a_m(\mathbf x)\ (n\neq m),\ 
  a_n(x_j)\to\infty\ (1\le j\le d)\Bigr\}.
\]
Accordingly, all dimension estimates are tail statements.
Since modifying finitely many initial digit vectors does not change Hausdorff dimension
(Lemma~\ref{lem:delete-finite}), the estimates are governed by the tail geometry of
cylinders.

To implement the covering argument in Section~\ref{sec:upper-bound}, we encode the anisotropic cylinder geometry by a singular-value potential. For
$\mathbf a=(a_1,\dots,a_d)\in\N^d$ let $a_{(1)}\le\cdots\le a_{(d)}$ denote the
nondecreasing rearrangement. For $s\in(0,d]$ and $k<s\le k+1$ we set
\[
  \phi^s(\mathbf a)
  :=
  \Bigl(\prod_{j=1}^k a_{(j)}^{-2}\Bigr)\, a_{(k+1)}^{-2(s-k)}.
\]
This potential yields a cylinder covering criterion and the associated critical threshold
\[
  s_\ast
  :=
  \inf\Bigl\{0<s\le d:\ \sum_{\mathbf a\in S}\phi^s(\mathbf a)<\infty\Bigr\},
\]
in particular, $s_\ast$ is the critical exponent governing our upper bounds.
We refer to Section~\ref{sec:upper-bound} for the covering argument and its equivalent
Dirichlet-series formulation.

\begin{theorem}[Universal upper bound and ambient sharpness]\label{thm:intro-universal}
For every infinite $S\subset\N^d$,
\[
  \dim_H(\mathcal E_S^{\mathrm{vec}})
  \le \dim_H(\mathcal E_S)\le \frac d2.
\]
For the full alphabet $S=\N^d$, both inequalities are equalities.
\end{theorem}

\begin{corollary}[Polynomial growth in the balanced regime]\label{cor:intro-growth}
Assume $S\subset\mathcal C_K$ for some $K\ge 1$ and $\#(S\cap Q_N)\ll N^{d/\alpha}$ for some
$\alpha\ge 1$. Then $\dim_H(\mathcal E_S)\le d/(2\alpha)$.
\end{corollary}

\begin{corollary}[Vector-injective balanced regime]\label{cor:intro-vec}
If $S\subset\mathcal C_K$ is uniformly $K$--balanced and has polynomial density
of exponent $\alpha\ge1$, then
\[
  \dim_H(\mathcal E_S^{\mathrm{vec}})
  =\dim_H(\mathcal E_S)
  =\frac{d}{2\alpha}.
\]
\end{corollary}

Theorem~\ref{thm:intro-universal} gives a universal upper bound and shows that it is
sharp for the full alphabet.  Corollary~\ref{cor:intro-growth} is the corresponding
balanced growth bound.  In the polynomial-density regime, the seed construction in
Section~\ref{sec:three-stage} supplies the matching lower bound inside the
vector-injective subregime, yielding Corollary~\ref{cor:intro-vec}.

\subsection{Organization of the paper}

Section~\ref{sec:prelim} collects the basic tools used throughout the paper.
In Subsection~\ref{subsec:cf-facts} we recall the required continued--fraction
estimates and the geometry of coordinatewise cylinders in $(0,1)^d$.
Subsection~\ref{subsec:density} fixes notation for upper density and upper Banach
density on $\N^d$, together with the relative versions used later.
Subsection~\ref{subsec:almost-lip} records the almost--Lipschitz / H\"older
stability statements for elimination maps, which are the quantitative input for
the surgery step.

Section~\ref{sec:upper-bound} develops the covering framework for digit--restricted
sets in $\N^d$.  We formulate the singular--value potential and the associated
Dirichlet--series criteria, and prove the universal upper bounds together with their
ambient and balanced sharpness statements, culminating in Theorems~\ref{thm:ubd-main}
and~\ref{thm:zeta-vs-singular} (together with the polynomial--growth consequences).

Section~\ref{sec:moran} develops higher--dimensional Moran constructions for
coordinatewise continued--fraction cylinders, providing the model families and
lower--bound mechanisms used later, compatible with coordinatewise divergence and, when
needed, vector--injectivity.

Section~\ref{sec:three-stage} implements the three--stage construction on $\N^d$.
Subsection~\ref{subsec:step1} produces a relatively thin subset with good bookkeeping
properties; Subsection~\ref{subsec:step2} constructs extreme seed sets in the complement
and computes their dimension; Subsection~\ref{subsec:step3} performs the digit insertion
and records the elimination estimates that transfer density information into digit sets
without loss of Hausdorff dimension.

Finally, Section~\ref{sec:proofs} assembles these ingredients to prove the main
dimension theorems and the transference principles stated in the introduction,
including the positive--density and positive--Banach--density regimes.

\section{Preliminaries}\label{sec:prelim}

This section gathers several technical lemmas that will be used repeatedly in the
subsequent proofs.  Their role is to provide quantitative control over density notions
on $\mathbb N^d$, as well as the basic geometric estimates for coordinatewise
continued--fraction cylinders.

Compared with the one--dimensional setting, additional care is required in higher
dimensions to handle translations of density boxes, anisotropic growth of digit
vectors, and the fact that the geometry of multidimensional cylinders is governed by
the largest coordinate scale.  The results collected here allow us to reduce to
well--positioned and uniformly balanced regimes, and to carry out the later
constructions without losing control of density or Hausdorff dimension.

\subsection{Some facts on continued fractions}
\label{subsec:cf-facts}

We briefly recall a few standard facts about regular continued fractions that
will be used throughout the paper. For each $n \in \mathbb{N}$ and $(a_1, \ldots, a_n) \in \mathbb{N}^n$, define the convergents $p_n/q_n$ recursively by:
\[
\begin{aligned}
p_{-1} &= 1, \quad p_0 = 0, \quad p_i = a_i p_{i-1} + p_{i-2} \quad \text{for } i=1, \ldots, n, \\
q_{-1} &= 0, \quad q_0 = 1, \quad q_i = a_i q_{i-1} + q_{i-2} \quad \text{for } i=1, \ldots, n.
\end{aligned}
\]

The $n$-th fundamental interval is defined by:
\[
I(a_1, \ldots, a_n) = 
\begin{cases}
\left[\dfrac{p_n}{q_n}, \dfrac{p_n+p_{n-1}}{q_n+q_{n-1}}\right) & \text{if } n \text{ is even}, \\[10pt]
\left(\dfrac{p_n+p_{n-1}}{q_n+q_{n-1}}, \dfrac{p_n}{q_n}\right] & \text{if } n \text{ is odd}.
\end{cases}
\]

This interval consists of all $x \in (0,1)$ whose regular continued fraction expansion begins with $a_1, \ldots, a_n$.

The following classical estimates provide control over the geometry of these intervals.
Both estimates follow from standard arguments in continued fraction theory.

\begin{lemma}\label{lem:1d_geometry} 
For every $n\in\N$ and $(a_1,\dots,a_n)\in\N^n$, the following hold.
\begin{enumerate}[label=(\roman*)]
\item
\[
\frac12 \prod_{i=1}^n \frac{1}{(a_i+1)^2}
\;\le\;
|I(a_1,\dots,a_n)|
\;\le\;
\prod_{i=1}^n \frac{1}{a_i^2}.
\]

\item
For any two irrational numbers $x_1,x_2\in(0,1)$ with the same first $n$ partial
quotients $a_1,\dots,a_n$ but
\[
|a_{n+1}(x_1)-a_{n+1}(x_2)|\ge 2,
\]
there exists an absolute constant $c>0$ such that
\[
|x_1-x_2|
\;\ge\;
\frac{c}{q_n^2}\,
\frac{|a_{n+1}(x_1)-a_{n+1}(x_2)|}
{a_{n+1}(x_1)a_{n+1}(x_2)}.
\]
\end{enumerate}
\end{lemma}

\begin{proof}
Both statements are standard consequences of the elementary geometry of
continued fraction cylinders and can be verified by a direct computation.
See, for instance, \cite{EinsiedlerWard2011,Good1941,Hirst1973,Nakajima2025}.
\end{proof}

\begin{lemma}\label{lem:I-vs-qn}
For every $n\ge 1$ and every word $(a_1,\dots,a_n)\in\N^n$ with convergents $p_n/q_n$,
\[\frac{1}{2q_n^2}\le 
|I(a_1,\dots,a_n)|
=
\frac{1}{q_n(q_n+q_{n-1})}
\le \frac{1}{q_n^2}.
\]

\end{lemma}

\begin{proof}
See, for example, \cite{EinsiedlerWard2011,Good1941,Hirst1973,Nakajima2025}.
\end{proof}

We now extend these concepts to higher dimensions. For
$\mathbf{x}=(x_1,\ldots,x_d)\in((0,1)\setminus\Q)^d$, define the digit vector at time
$i$ by
\[
\mathbf{a}_i(\mathbf{x})=(a_i(x_1),\ldots,a_i(x_d))\in\mathbb{N}^d,
\]
where $a_i(x_j)$ denotes the $i$th partial quotient of $x_j$.

\begin{definition}
For $\omega=(\mathbf{a}_1,\ldots,\mathbf{a}_n)\in(\mathbb{N}^d)^n$, define
\[
I_\omega=I(\mathbf{a}_1,\ldots,\mathbf{a}_n)
:=\prod_{j=1}^d I(a_1^{(j)},\ldots,a_n^{(j)}).
\]
Its points with irrational coordinates are exactly the $\mathbf{x}$ satisfying
$\mathbf{a}_i(\mathbf{x})=\mathbf{a}_i$ for $1\le i\le n$.
\end{definition}

The geometric structure of these cylinder sets is characterized by their product decomposition:

\begin{lemma}\label{lem:cylinder_structure}
For $\omega=(\mathbf{a}_1,\ldots,\mathbf{a}_n)\in(\mathbb{N}^d)^n$, the cylinder set decomposes as
\[
I_\omega=I^{(1)}_\omega\times I^{(2)}_\omega\times\cdots\times I^{(d)}_\omega,
\]
where $I^{(j)}_\omega=I(a_1^{(j)},\ldots,a_n^{(j)})$.
\end{lemma}

\begin{proof}
This is immediate from the definition of $I_\omega$.
\end{proof}

This product structure allows us to establish diameter estimates for multidimensional cylinders:

\begin{proposition}\label{prop:cylinder_estimate}
Let $\omega=(\mathbf a_1,\dots,\mathbf a_n)\in(\N^d)^n$. Then there exist constants
$c_1,c_2>0$, depending only on $d$, such that
\[
c_1 \max_{1\le j\le d}\ \prod_{i=1}^n (a_i^{(j)}+1)^{-2}
\;\le\;
\diam(I_\omega)
\;\le\;
c_2 \max_{1\le j\le d}\ \prod_{i=1}^n (a_i^{(j)})^{-2}.
\]
\end{proposition}

\begin{proof}
By Lemma~\ref{lem:cylinder_structure},
$I_\omega=\prod_{j=1}^d I^{(j)}_\omega$ with $I^{(j)}_\omega=I(a_1^{(j)},\dots,a_n^{(j)})$.
Hence
\[
\max_{1\le j\le d}|I^{(j)}_\omega|
\;\le\;
\diam(I_\omega)
\;\le\;
\sqrt d\,\max_{1\le j\le d}|I^{(j)}_\omega|.
\]
Applying Lemma~\ref{lem:1d_geometry}(i) coordinatewise gives
\[
\frac12\prod_{i=1}^n (a_i^{(j)}+1)^{-2}
\;\le\;
|I^{(j)}_\omega|
\;\le\;
\prod_{i=1}^n (a_i^{(j)})^{-2}.
\]
Taking the maximum over $j$ and absorbing the factors $\frac12$ and $\sqrt d$ into
$c_1,c_2$ yields the claim.
\end{proof}

\begin{lemma}\label{lem:cylinder_digit_separation}
Let $\mathbf{x}_1,\mathbf{x}_2\in(0,1)^d$ satisfy $\mathbf{a}_i(\mathbf{x}_1)=\mathbf{a}_i(\mathbf{x}_2)$
for $1\le i\le n$ and
\[
\bigl\|\mathbf{a}_{n+1}(\mathbf{x}_1)-\mathbf{a}_{n+1}(\mathbf{x}_2)\bigr\|_\infty\ge 2.
\]
Let $\mathcal D:=\{1\le j\le d:\ |a_{n+1}^{(j)}(\mathbf{x}_1)-a_{n+1}^{(j)}(\mathbf{x}_2)|\ge 2\}$.
Then there exists a constant $C>0$ depending only on $d$ such that
\[
\|\mathbf{x}_1-\mathbf{x}_2\|_\infty
\ge
C\max_{j\in\mathcal D}
\left(
|I(a_1^{(j)},\ldots,a_n^{(j)})|
\cdot
\frac{|a_{n+1}^{(j)}(\mathbf{x}_1)-a_{n+1}^{(j)}(\mathbf{x}_2)|}
{a_{n+1}^{(j)}(\mathbf{x}_1)\,a_{n+1}^{(j)}(\mathbf{x}_2)}
\right).
\]
\end{lemma}

\begin{proof}
By Lemma~\ref{lem:cylinder_structure},
\[
I(\mathbf{a}_1,\ldots,\mathbf{a}_n)=\prod_{j=1}^d I(a_1^{(j)},\ldots,a_n^{(j)}).
\]
Fix $j\in\mathcal D$. Then $x_{1,j}$ and $x_{2,j}$ have the same first $n$ partial quotients
$(a_1^{(j)},\ldots,a_n^{(j)})$ and
$|a_{n+1}(x_{1,j})-a_{n+1}(x_{2,j})|\ge 2$. By Lemma~\ref{lem:1d_geometry}{\rm (ii)},
\[
|x_{1,j}-x_{2,j}|
\ge
\frac{c_0}{q_{n,j}^2}\,
\frac{|a_{n+1}(x_{1,j})-a_{n+1}(x_{2,j})|}
{a_{n+1}(x_{1,j})\,a_{n+1}(x_{2,j})},
\]
where $q_{n,j}$ is the denominator of the $n$-th convergent associated with
$(a_1^{(j)},\ldots,a_n^{(j)})$. Since $|I(a_1^{(j)},\ldots,a_n^{(j)})|\le q_{n,j}^{-2}$
by Lemma~\ref{lem:I-vs-qn}, we obtain
\[
|x_{1,j}-x_{2,j}|
\ge
c_0\,|I(a_1^{(j)},\ldots,a_n^{(j)})|
\cdot
\frac{|a_{n+1}(x_{1,j})-a_{n+1}(x_{2,j})|}
{a_{n+1}(x_{1,j})\,a_{n+1}(x_{2,j})}.
\]
Taking the maximum over $j\in\mathcal D$ and using
$\|\mathbf{x}_1-\mathbf{x}_2\|_\infty=\max_{1\le j\le d}|x_{1,j}-x_{2,j}|$ gives the claim
(with $C=c_0$).
\end{proof}

The next lemma gives a uniform quantitative separation estimate between sibling
cylinders, and will be used to verify the separation condition in
Definition~\ref{def:moran-nd}.

\begin{lemma}[Sibling cylinder separation]\label{lem:sibling_cylinder_separation}
Let $n\ge 1$ and let $\omega,\omega'\in(\N^d)^n$ be distinct words with the same prefix
of length $n-1$, say
\[
\omega=(\mathbf a_1,\dots,\mathbf a_{n-1},\mathbf v),\qquad
\omega'=(\mathbf a_1,\dots,\mathbf a_{n-1},\mathbf w).
\]
Assume that $\|\mathbf v-\mathbf w\|_\infty\ge 2$. Then $\overline{I_\omega}$ and
$\overline{I_{\omega'}}$ are disjoint and there exists an absolute constant $c>0$
(say $c=\tfrac14$) such that
\[
\dist(I_\omega,I_{\omega'})
\;\ge\;
c\,
\min_{1\le j\le d}\min\bigl(|I^{(j)}_\omega|,\ |I^{(j)}_{\omega'}|\bigr).
\]
\end{lemma}

\begin{proof}
By Lemma~\ref{lem:cylinder_structure},
\[
I_\omega=\prod_{j=1}^d I^{(j)}_\omega,\qquad I_{\omega'}=\prod_{j=1}^d I^{(j)}_{\omega'}.
\]
Choose $j_0$ such that $|v^{(j_0)}-w^{(j_0)}|\ge 2$. Let
\[
\alpha:=(a_1^{(j_0)},\dots,a_{n-1}^{(j_0)}),\qquad
m:=v^{(j_0)},\quad m':=w^{(j_0)}.
\]
Then
\[
I^{(j_0)}_\omega=I(\alpha,m),\qquad I^{(j_0)}_{\omega'}=I(\alpha,m').
\]
Assume w.l.o.g.\ $m<m'$, so $m'\ge m+2$. The one-dimensional depth-$n$ cylinders
$\{I(\alpha,t)\}_{t\ge 1}$ form a disjoint partition of the parent interval $I(\alpha)$
(up to endpoints) in their natural order, hence the intermediate cylinder
$I(\alpha,m+1)$ lies between $I(\alpha,m)$ and $I(\alpha,m')$. Therefore
\[
\dist\bigl(I(\alpha,m),I(\alpha,m')\bigr)\ \ge\ |I(\alpha,m+1)|.
\]

Write $p_k/q_k$ for the convergents associated with the prefix $\alpha$.
For the word $(\alpha,t)$ one has $q_n(t)=tq_{n-1}+q_{n-2}$, and the endpoint formula
for continued-fraction cylinders gives 
\[
|I(\alpha,t)|=\frac{1}{q_n(t)\bigl(q_n(t)+q_{n-1}\bigr)}
=\frac{1}{(tq_{n-1}+q_{n-2})\bigl((t+1)q_{n-1}+q_{n-2}\bigr)}.
\]
In particular, $t\mapsto |I(\alpha,t)|$ is decreasing, so
\[
|I(\alpha,m+1)|\ \ge\ |I(\alpha,m')|
=\min\bigl(|I(\alpha,m)|,\ |I(\alpha,m')|\bigr).
\]
Hence
\[
\dist\bigl(I^{(j_0)}_\omega,I^{(j_0)}_{\omega'}\bigr)
\ \ge\ \min\bigl(|I^{(j_0)}_\omega|,\ |I^{(j_0)}_{\omega'}|\bigr).
\]
Finally, for Cartesian products we have
\[
\dist(I_\omega,I_{\omega'})\ \ge\ \dist\bigl(I^{(j_0)}_\omega,I^{(j_0)}_{\omega'}\bigr),
\]
which yields the stated bound (with $c=1$).
\end{proof}

These geometric separation properties will be used repeatedly in the
construction and analysis of higher-dimensional fractal transference
arguments.

To verify the vanishing mesh condition in Definition~\ref{def:moran-nd}, we record the following classical lemma.

\begin{lemma}[Uniform exponential decay of continued-fraction cylinders]
\label{lem:uniform-cylinder-decay}
There exist constants $C>0$ and $\vartheta\in(0,1)$ such that for every $n\ge1$ and
every word $(a_1,\dots,a_n)\in\N^n$,
\[
|I(a_1,\dots,a_n)| \le C\,\vartheta^n.
\]
In particular one may take $\vartheta=\varphi^{-2}$, where $\varphi=(1+\sqrt5)/2$.
\end{lemma}

\begin{proof}
Let $p_n/q_n$ be the $n$-th convergent of $(a_1,\dots,a_n)$. The endpoint formula gives
\[
|I(a_1,\dots,a_n)|=\frac{1}{q_n(q_n+q_{n-1})}\le \frac{1}{q_n^2}.
\]
Since $a_k\ge1$, the recursion $q_n=a_nq_{n-1}+q_{n-2}$ implies $q_n\ge q_{n-1}+q_{n-2}$, hence
$q_n\ge F_{n+1}$ for all $n$, where $(F_n)$ is the Fibonacci sequence. Moreover,
$F_{n+1}\ge \varphi^{\,n-1}$ for $n\ge1$ (by induction using $\varphi^2=\varphi+1$). Therefore
\[
|I(a_1,\dots,a_n)|\le F_{n+1}^{-2}\le \varphi^{2}\,\varphi^{-2n}.
\]
Thus one may take $\vartheta=\varphi^{-2}$ and $C=\varphi^{2}$.
\end{proof}

\begin{remark}\label{rmk:uniform-cylinder-decay-d}
  Applying Lemma~\ref{lem:uniform-cylinder-decay} coordinatewise, if
$I_\omega=\prod_{j=1}^d I^{(j)}_\omega$ is a $d$--dimensional cylinder of word length
$n$, then
\[
\diam(I_\omega)\le \sqrt d\,\max_{1\le j\le d}|I^{(j)}_\omega|
\le \sqrt d\,C\,\vartheta^n.
\]
In particular, any choice of cubes $Q_\omega\subset I_\omega$ at level $n$ satisfies
$\max_\omega \diam(Q_\omega)\to0$ as $n\to\infty$.
\end{remark}

\subsection{Scaling and reduction lemmas in \texorpdfstring{$\mathbb{N}^d$}{Nd}}
  \label{subsec:density}

The following lemma is used in the root set construction to control how polynomial
density behaves under changes of scale.

\begin{lemma}\label{lem:ratio-asymptotic}
If $S\subset\N^d$ has polynomial density of exponent $\alpha\ge1$, then for all $t,L>1$ and all sufficiently large $n$,
\[
\frac{|S\cap Q_{t^n}|}{|S\cap Q_{Lt^n}|}\asymp L^{-d/\alpha}.
\]
The implied constants are independent of $t$ and $L$, although the threshold for $n$
may depend on these parameters.
\end{lemma}
\begin{proof}
By polynomial density, for all large $R$,
\[
C_1 \frac{R^{d/\alpha}}{(\log R)^\beta}\le |S\cap Q_R|\le
C_2 \frac{R^{d/\alpha}}{(\log R)^\beta}.
\]
Apply this with $R=t^n$ and $R=Lt^n$ to get
\[
\frac{C_1}{C_2}\,L^{-d/\alpha}\Bigl(\frac{\log t^n}{\log(Lt^n)}\Bigr)^\beta
\le
\frac{|S\cap Q_{t^n}|}{|S\cap Q_{Lt^n}|}
\le
\frac{C_2}{C_1}\,L^{-d/\alpha}\Bigl(\frac{\log(Lt^n)}{\log t^n}\Bigr)^\beta.
\]
Since $\log(Lt^n)=n\log t+\log L$, we have
\[
\Bigl(\frac{\log(Lt^n)}{\log t^n}\Bigr)^\beta=\Bigl(1+\frac{\log L}{n\log t}\Bigr)^\beta\to1
\quad(n\to\infty),
\]
and similarly for its reciprocal. Hence for all sufficiently large $n$ these log-factors are
bounded above and below by absolute constants (depending only on $\beta$), yielding
\[
C_1' L^{-d/\alpha}\le \frac{|S\cap Q_{t^n}|}{|S\cap Q_{Lt^n}|}\le C_2' L^{-d/\alpha},
\]
with $C_1',C_2'$ independent of $t,L$.
\end{proof}

The following lemma is used in the proof of Theorem~\ref{thm:Zdd-fractal-transference}, part~(b).
It produces a subset $S'\subset S$ with $\dB(S')=\dB(S)$ and $\dbar(S')=0$, witnessed along translates $Q_{N_k}(\mathbf v_k)$ with $\min_i (v_k)_i\to\infty$.

\begin{lemma}[Density reduction with diverging translations]\label{lem:density-reduction}
Let $S\subset\N^d$ and set $\alpha:=\dB(S)$.
Then there exist a strictly increasing sequence $(N_k)_{k\ge1}$ and vectors
$(\mathbf v_k)_{k\ge1}\subset\N^d$, and a subset $S'\subset S$ such that
\[
  \dB(S')=\dB(S)=\alpha
  \qquad\text{and}\qquad
  \dbar(S')=0,
\]
and such that
\[
  \frac{|S'\cap Q_{N_k}(\mathbf v_k)|}{|Q_{N_k}|}\longrightarrow \alpha
  \qquad (k\to\infty).
\]
Moreover, the boxes $Q_{N_k}(\mathbf v_k)$ are pairwise disjoint, and
\[
  \min_{1\le i\le d}(v_k)_i \longrightarrow \infty.
\]
\end{lemma}

\begin{proof}
We proceed in three steps. First, by partitioning a large cube where $S$ has density close to $\alpha$ and averaging over subcubes, we obtain cubes $Q_{N_k}(\mathbf v_k)$ with $|S\cap Q_{N_k}(\mathbf v_k)|/|Q_{N_k}|\to\alpha$ and $\min_i(v_k)_i\to\infty$. Second, we construct $S'\subset S$ by restricting $S$ to a rapidly separated sequence of such cubes. Third, we verify that $\dB(S')=\alpha$, $\dbar(S')=0$, and that the chosen cubes witness the required convergence.

	Fix $\varepsilon_k:=1/k$ and set $N_k:=k$, $M_k:=k^2$.
	By $\dB(S)=\alpha$, choose $R_k\in\N$ and $\mathbf u_k\in\mathbb Z_{\ge0}^d$ such that
	\[
	\frac{|S\cap Q_{R_k}(\mathbf u_k)|}{|Q_{R_k}|}>\alpha-\varepsilon_k,
	\]
	and pass to a subsequence so that $R_k\ge k^4$.
	
	\smallskip
	\noindent\textbf{Step 1: far--out witnesses.}
	Let $q_k:=\lfloor R_k/N_k\rfloor$ and $R_k':=q_kN_k$, so $0\le R_k-R_k'<N_k$ and
	\[
	|Q_{R_k}|-|Q_{R_k'}|=R_k^d-(R_k')^d\le d(R_k-R_k')R_k^{d-1}\le dN_kR_k^{d-1}.
	\]
	Hence
	\[
	\frac{|S\cap Q_{R_k'}(\mathbf u_k)|}{|Q_{R_k'}|}
	\ge
	\frac{|S\cap Q_{R_k}(\mathbf u_k)|-(|Q_{R_k}|-|Q_{R_k'}|)}{|Q_{R_k'}|}
	>
	(\alpha-\varepsilon_k)-\frac{dN_kR_k^{d-1}}{(R_k')^d}.
	\]
	Since $R_k'\ge R_k-N_k$ and $R_k\ge k^4$, the last term is $\le \varepsilon_k$ for all large $k$.
	Thus, after discarding finitely many $k$,
	\[
	\frac{|S\cap Q_{R_k'}(\mathbf u_k)|}{|Q_{R_k'}|}>\alpha-2\varepsilon_k.
	\]
	
	Partition $Q_{R_k'}(\mathbf u_k)$ into the $q_k^d$ disjoint cubes
	$Q_{N_k}(\mathbf u_k+N_k\mathbf m)$, $\mathbf m\in\{0,\dots,q_k-1\}^d$.
	Let $\mathcal I_k$ be the set of \emph{$M_k$--interior} indices, i.e.
	$\min_i((u_k)_i+N_km_i)\ge M_k$, and let $\mathcal B_k:=\mathcal I_k^c$.
	For each coordinate $i$, there are at most $\lceil M_k/N_k\rceil\le k+1$ boundary layers, hence
	\[
	|\mathcal B_k|\le d(k+1)q_k^{d-1}.
	\]
	Since $q_k\ge R_k/N_k-1\ge k^3-1$, we have $|\mathcal B_k|/q_k^d\le 2d/k^2$ for all large $k$.
	
	For $\mathbf m$ define
	\[
	\rho_{\mathbf m}:=\frac{|S\cap Q_{N_k}(\mathbf u_k+N_k\mathbf m)|}{|Q_{N_k}|}\in[0,1].
	\]
	Averaging over all $q_k^d$ cubes gives
	\[
	\frac1{q_k^d}\sum_{\mathbf m}\rho_{\mathbf m}
	=
	\frac{|S\cap Q_{R_k'}(\mathbf u_k)|}{|Q_{R_k'}|}
	>
	\alpha-2\varepsilon_k.
	\]
	Using $\sum_{\mathbf m\in\mathcal B_k}\rho_{\mathbf m}\le |\mathcal B_k|$ and
	$|\mathcal I_k|=q_k^d-|\mathcal B_k|$, we get
	\[
	\frac1{|\mathcal I_k|}\sum_{\mathbf m\in\mathcal I_k}\rho_{\mathbf m}
	\ge
	\frac{q_k^d(\alpha-2\varepsilon_k)-|\mathcal B_k|}{q_k^d-|\mathcal B_k|}
	\ge
	\alpha-2\varepsilon_k-\frac{|\mathcal B_k|}{q_k^d-|\mathcal B_k|}
	\ge
	\alpha-2\varepsilon_k-\frac{4d}{k^2}
	\]
	for all large $k$. Hence there exists $\mathbf m_k\in\mathcal I_k$ with
	\[
	\rho_{\mathbf m_k}\ge \alpha-2\varepsilon_k-\frac{4d}{k^2}.
	\]
	Set $\mathbf v_k:=\mathbf u_k+N_k\mathbf m_k$. Then $\min_i(v_k)_i\ge M_k\to\infty$.  The preceding estimate gives
		$\liminf_k\rho_{\mathbf m_k}\ge\alpha$, while
		\[
		\rho_{\mathbf m_k}
		\le \sup_{\mathbf v\in\mathbb Z_{\ge0}^d}
		\frac{|S\cap Q_{N_k}(\mathbf v)|}{|Q_{N_k}|}
		\]
		and the definition of $\dB(S)$ gives $\limsup_k\rho_{\mathbf m_k}\le\alpha$.
		Consequently,
		\[
		\frac{|S\cap Q_{N_k}(\mathbf v_k)|}{|Q_{N_k}|}=\rho_{\mathbf m_k}\longrightarrow \alpha.
		\]
	
	\smallskip
	\noindent\textbf{Step 2: define $S'$.}
	Choose an increasing subsequence $k_j$ recursively so that, with
		\[
		\mathcal B_j:=Q_{N_{k_j}}(\mathbf v_{k_j}),
		\]
		one has
		\[
		\min_i(\mathbf v_{k_j})_i\ge 10^jN_{k_j},
		\qquad
		\min_i(\mathbf v_{k_{j+1}})_i>
		\max_i(\mathbf v_{k_j})_i+N_{k_j}.
		\]
		This is possible because $N_k=k$ and
		$\min_i(\mathbf v_k)_i\ge k^2$, after passing to the subsequence already used
		above.  In particular, the boxes $\mathcal B_j$ are pairwise disjoint and ordered
		away from the origin.  Define
		\[
		S':=\bigcup_{j\ge1}\bigl(S\cap \mathcal B_j\bigr)\subset S.
		\]
	
	\smallskip
	\noindent\textbf{Step 3: verify densities.}
	Since $S'\cap \mathcal B_j=S\cap \mathcal B_j$,
	\[
	\dB(S')\ge \limsup_{j\to\infty}\frac{|S'\cap \mathcal B_j|}{|\mathcal B_j|}
	=\limsup_{j\to\infty}\frac{|S\cap \mathcal B_j|}{|\mathcal B_j|}=\alpha,
	\]
	and $S'\subset S$ gives $\dB(S')\le \dB(S)=\alpha$, hence $\dB(S')=\alpha$.
	
	For $\dbar(S')$, let $N$ be large and let $j=j(N)$ be maximal with $\mathcal B_j\cap Q_N\neq\emptyset$.
	Then $S'\cap Q_N$ meets only $\mathcal B_1,\dots,\mathcal B_j$, so
	\[
	|S'\cap Q_N|\le \sum_{m\le j}|\mathcal B_m|\le j\,N_{k_j}^d.
	\]
	Moreover, $\mathcal B_j\cap Q_N\neq\emptyset$ implies
	$N\ge\max_i(\mathbf v_{k_j})_i+1\ge\min_i(\mathbf v_{k_j})_i$, hence
	\[
	\frac{|S'\cap Q_N|}{|Q_N|}
	\le
	\frac{j\,N_{k_j}^d}{N^d}
	\le
	\frac{j\,N_{k_j}^d}{(\frac12\min_i(\mathbf v_{k_j})_i)^d}
	\le
	2^d j\Bigl(\frac{N_{k_j}}{\min_i(\mathbf v_{k_j})_i}\Bigr)^d
	\le
	2^d j\,10^{-dj}\xrightarrow[N\to\infty]{}0.
	\]
	Thus $\dbar(S')=0$.  Relabelling $N_{k_j}$ and $\mathbf v_{k_j}$ as
	$N_j$ and $\mathbf v_j$, respectively, the cubes $\mathcal B_j$ satisfy all the
	conclusions of the lemma.
\end{proof}

The next two lemmas quantify the abundance of uniformly balanced vectors.  They will
be used in the transference constructions, where coordinatewise divergence must be
recovered from norm growth.

 \begin{lemma}[Density of $\mathcal C_K$ in $Q_N$]\label{lem:Ck-density}
Fix $d\ge 1$ and $K\ge 1$. Let
\[
Q_N:=[1,N]^d\cap\N^d,
\qquad
\mathcal C_K
:=
\Bigl\{\mathbf a\in\N^d:\ \frac{\max_j a_j}{\min_j a_j}\le K\Bigr\}.
\]
Then
\[
\lim_{N\to\infty}\frac{|\mathcal C_K\cap Q_N|}{|Q_N|}
=
\Bigl(1-\frac1K\Bigr)^{d-1}.
\]
For $d=1$, the right-hand side is understood as $1$.
Equivalently,
\[
\lim_{N\to\infty}\frac{|Q_N\setminus \mathcal C_K|}{|Q_N|}
=
1-\Bigl(1-\frac1K\Bigr)^{d-1}.
\]
\end{lemma}

 \begin{proof}
If $d=1$, then $\mathcal C_K=\N$ for every $K\ge1$, so the limit is $1$.
If $d\ge2$ and $K=1$, then $\mathcal C_1\cap Q_N$ is the diagonal and has cardinality
$N$, so the limit is $0$.  We may therefore assume $d\ge2$ and $K>1$.

Write $|Q_N|=N^d$ and put $M:=\lfloor N/K\rfloor$.  For $\mathbf a\in Q_N$ let
$m:=\min_j a_j$.  The condition $\mathbf a\in\mathcal C_K$ is equivalent to
$\max_j a_j\le Km$, i.e.\ $\mathbf a\in[m,U_m]^d$ with
$U_m:=\min\{\lfloor Km\rfloor,N\}$.  Let $L_m:=U_m-m+1$.
The number of vectors in $\mathcal C_K\cap Q_N$ with minimum coordinate $m$ is
$L_m^d-(L_m-1)^d$, whence
\[
|\mathcal C_K\cap Q_N|
=\sum_{m=1}^M\Bigl((\lfloor Km\rfloor-m+1)^d-(\lfloor Km\rfloor-m)^d\Bigr)
 +(N-M)^d.
\]
Since $\lfloor Km\rfloor-m=(K-1)m+O_K(1)$, the binomial theorem and
$\sum_{m\le M}m^{d-1}=M^d/d+O_d(M^{d-1})$ give
\[
\frac{|\mathcal C_K\cap Q_N|}{N^d}
=(K-1)^{d-1}\Bigl(\frac MN\Bigr)^d
 +\Bigl(1-\frac MN\Bigr)^d+O_{d,K}(N^{-1}).
\]
Since $M/N\to1/K$, the limit is
\[
(K-1)^{d-1}K^{-d}+(1-K^{-1})^d=(1-K^{-1})^{d-1}.
\]
\end{proof}
 
\begin{lemma}[Positive density forces bounded anisotropy]\label{lem:balanced-density}
Let $A\subset\N^d$ satisfy
\[
\overline d(A):=\limsup_{N\to\infty}\frac{|A\cap Q_N|}{|Q_N|}>0,
\qquad
Q_N=[1,N]^d\cap\N^d.
\]
Then there exists $K\ge 1$ such that
\[
\overline d\bigl(A\cap\mathcal C_K\bigr)>0,
\qquad
\mathcal C_K=\Bigl\{\mathbf a\in\N^d:\frac{\max_j a_j}{\min_j a_j}\le K\Bigr\}.
\]
\end{lemma}

\begin{proof}
	Let $\delta:=\overline d(A)>0$.  If $d=1$, take $K=1$; then
		$A\cap\mathcal C_1=A$, and the conclusion is immediate.  Assume henceforth that
		$d\ge2$ and choose $K>1$ so that
	\[
	1-\Bigl(1-\frac1K\Bigr)^{d-1}<\frac{\delta}{4}.
	\]
	By Lemma~\ref{lem:Ck-density}, there exists $N_0$ such that for all $N\ge N_0$,
	\[
	\frac{|Q_N\setminus\mathcal C_K|}{|Q_N|}<\frac{\delta}{2}.
	\]
	Pick $N_i\to\infty$ with $|A\cap Q_{N_i}|/|Q_{N_i}|\to\delta$ and $N_i\ge N_0$, so for large $i$,
	$|A\cap Q_{N_i}|/|Q_{N_i}|>3\delta/4$.  Then
	\[
	\frac{|A\cap\mathcal C_K\cap Q_{N_i}|}{|Q_{N_i}|}
	\ge
	\frac{|A\cap Q_{N_i}|}{|Q_{N_i}|}-\frac{|Q_{N_i}\setminus\mathcal C_K|}{|Q_{N_i}|}
	\ge \frac{3\delta}{4}-\frac{\delta}{2}=\frac{\delta}{4}.
	\]
	Taking $\limsup_i$ gives $\overline d(A\cap\mathcal C_K)\ge \delta/4>0$.
\end{proof}

\subsection{Almost Lipschitz maps and elimination maps}
\label{subsec:almost-lip}

We recall the notion of almost Lipschitz maps, which will be used to control
the possible loss of Hausdorff dimension when digits are eliminated from
symbolic addresses.

\begin{definition}
Let $(X,d)$ be a compact metric space, $F\subset X$, and $f:F\to X$.
We say that
\begin{itemize}[leftmargin=2em]
\item
$f$ is \emph{Hölder continuous with exponent $\gamma\in(0,1]$} if there exists
$C>0$ such that
\[
d\bigl(f(x),f(y)\bigr)\le C\, d(x,y)^\gamma
\quad\text{for all }x,y\in F;
\]
\item
$f$ is \emph{almost Lipschitz} if it is Hölder continuous with exponent
$\gamma$ for every $\gamma\in(0,1)$.
\end{itemize}
\end{definition}

The following dimension estimate is standard.

\begin{lemma}\label{lem:almost-lip}
Let $(X,d)$ be a compact metric space, $F\subset X$, and let $f:F\to X$ be an
almost Lipschitz map.
Then
\[
\haus f(F)\le \haus F .
\]
\end{lemma}

\begin{proof}
See \cite[Proposition~3.3]{Falconer2014}.
\end{proof}

In the sequel, we construct \emph{elimination maps} which erase prescribed
digits from symbolic addresses.
Lemma~\ref{lem:almost-lip} ensures that these operations do not increase the
Hausdorff dimension.

\section{Upper bounds on \texorpdfstring{$\dim_H$}{Hausdorff dimension} for digit--restricted sets in \texorpdfstring{$\N^d$}{Nd}}
\label{sec:upper-bound}

Let $S\subset\N^d$. Recall that
\[
  \widetilde{\mathcal E}_S
  :=
  \{\mathbf x\in((0,1)\setminus\Q)^d:\ \mathbf a_n(\mathbf x)\in S \text{ for all } n\ge1\},
\]
and its escaping part
\[
  \mathcal E_S
  :=
  \Bigl\{\mathbf x\in\widetilde{\mathcal E}_S:\ a_n(x_j)\to\infty \text{ for each } 1\le j\le d\Bigr\}.
\]
We also consider the vector-injective subregime
\[
  \mathcal E_S^{\mathrm{vec}}
  :=
  \Bigl\{\mathbf x\in\widetilde{\mathcal E}_S:\ \mathbf a_n(\mathbf x)\neq \mathbf a_m(\mathbf x)\ (n\neq m),\
  a_n(x_j)\to\infty\ (1\le j\le d)\Bigr\}.
\]
Clearly, $\mathcal E_S^{\mathrm{vec}}\subset \mathcal E_S\subset \widetilde{\mathcal E}_S$.
Hence any upper bound for $\dim_H(\widetilde{\mathcal E}_S)$ yields an upper bound for
$\dim_H(\mathcal E_S)$ and $\dim_H(\mathcal E_S^{\mathrm{vec}})$, while any lower bound for
$\dim_H(\mathcal E_S^{\mathrm{vec}})$ (or $\dim_H(\mathcal E_S)$) also yields a lower bound for
$\dim_H(\widetilde{\mathcal E}_S)$.

\subsection{Main upper bounds}
\label{subsec:upper-main-statements}

We collect the main upper bounds proved in this section. Proofs are given in the
subsections that follow.

\begin{theorem}[Universal upper bound and ambient sharpness]\label{thm:ubd-main}
For every infinite $S\subset\N^d$,
\[
  \dim_H(\mathcal E_S^{\mathrm{vec}})
  \le \dim_H(\mathcal E_S)
  \le \frac d2.
\]
Moreover,
\[
  \dim_H(\mathcal E_{\N^d}^{\mathrm{vec}})
  =\dim_H(\mathcal E_{\N^d})
  =\frac d2.
\]
\end{theorem}

\begin{remark}\label{rem:product-structure}
Assume that $S=S_1\times\cdots\times S_d\subset\N^d$. Then
\[
  \mathcal E_S=\prod_{j=1}^d \mathcal E_{S_j},
\]
and the standard product inequalities give
\[
  \sum_{j=1}^d \dim_H(\mathcal E_{S_j})
  \le
  \dim_H(\mathcal E_S)
  \le
  \sum_{j=1}^d \dim_P(\mathcal E_{S_j}).
\]
Thus $\dim_H(\mathcal E_S)\ge\sum_j\dim_H(\mathcal E_{S_j})$, with equality whenever
$\dim_P(\mathcal E_{S_j})=\dim_H(\mathcal E_{S_j})$ for every $j$.
\end{remark}

Fix an integer $N \ge 2$. Define
\[
E_{\ge N}:=\bigl\{x\in(0,1)\setminus\Q: a_n(x)\ge N \text{ for all }n\ge1\bigr\},
\qquad
\widetilde{\mathcal{E}}_N := (E_{\ge N})^d\subset((0,1)\setminus\Q)^d.
\]

\begin{theorem}[Asymptotics for the full restriction $a_n\ge N$]
\label{thm:full-restriction}
For every $\varepsilon>0$ there exists $N_0(\varepsilon)$ such that for all $N\ge N_0(\varepsilon)$,
\[
 \frac d2 \leq \dim_H(\widetilde{\mathcal{E}}_N) \le \frac d2 + \varepsilon.
\]
In particular, $\dim_H(\widetilde{\mathcal{E}}_N)\to d/2$ as $N\to\infty$.
\end{theorem}

  \begin{theorem} \label{thm:zeta-vs-singular}
  Let $S\subset\N^d$ be infinite. For $s\in(0,d]$ let $\phi^s$ be the singular value
  potential defined in \eqref{eq:def-phis} below, and set
  \begin{align*}
  \Phi_S(s)&:=\sum_{\mathbf a\in S}\phi^s(\mathbf a),\\
  s_\sharp&:=\inf\{0<s\le d:\Phi_S(s)<1\},\\
  s_\ast&:=\inf\{0<s\le d:\Phi_S(s)<\infty\}.
  \end{align*}
  where the infimum of the empty set is understood as $+\infty$.
  Define also the multivariate Dirichlet series
  \[
  \zeta_S(\boldsymbol{\sigma})
  :=
  \sum_{\mathbf a\in S}\prod_{j=1}^d a_j^{-\sigma_j}
  \quad(\boldsymbol{\sigma}\in(0,\infty)^d),
  \]
  its convergence domain $\mathcal D_S:=\{\boldsymbol{\sigma}:\zeta_S(\boldsymbol{\sigma})<\infty\}$,
  and the trace exponent
  \[
  \Lambda_S:=\inf_{\boldsymbol{\sigma}\in\mathcal D_S}(\sigma_1+\cdots+\sigma_d).
  \]
  Then:
  \begin{enumerate}[label=\textup{(\arabic*)},itemsep=2pt,topsep=2pt]
  \item $\dim_H(\widetilde{\mathcal{E}}_S)\le s_\sharp$.
    \item $\dim_H(\mathcal E_S^{\mathrm{vec}})\le\dim_H({\mathcal{E}}_S)\le s_\ast$  
  \item $\Lambda_S\le 2s_\ast$ (hence $\tfrac12\Lambda_S\le s_\ast\le s_\sharp$).
  \item If $S$ is uniformly $K$--balanced (i.e.\ $\max_i a_i\le K\min_i a_i$ for all $\mathbf a\in S$),
  then $s_\ast=\tfrac12\Lambda_S$.
  \end{enumerate}
  \end{theorem}

\begin{corollary}[Growth bound in the balanced region]\label{cor:growth-balanced}
Let $K\ge 1$ and let $S\subset \mathcal C_K\subset\N^d$ be infinite. 
Assume that there exists $\alpha\ge1$ such that
\[
  |S\cap Q_N|\ \ll\ N^{d/\alpha}
  \qquad (N\to\infty),
  \qquad Q_N:=[1,N]^d\cap\N^d.
\]
Then
\[
  \dim_H(\mathcal E_S^{\mathrm{vec}})\le \dim_H(\mathcal{E}_S)\ \le\ \frac{d}{2\alpha}.
\]
\end{corollary}

\begin{proof}[{Proof of Corollary~\ref{cor:growth-balanced} assuming Theorem~\ref{thm:zeta-vs-singular}:}]

Fix $\varepsilon>0$ and set $\sigma:=\frac1\alpha+\varepsilon$ and
$\boldsymbol{\sigma}:=(\sigma,\dots,\sigma)$.
We show that $\boldsymbol{\sigma}\in\mathcal D_S$, i.e.\ $\zeta_S(\boldsymbol{\sigma})<\infty$.

Separate the finite set $S\cap Q_1$ and decompose the remainder into the dyadic
shells $Q_{2^{n+1}}\setminus Q_{2^n}$. Since
$|S\cap Q_{2^{n+1}}|\ll 2^{(n+1)d/\alpha}$ and $S\subset \mathcal C_K$, for every
$\mathbf a\in S\cap(Q_{2^{n+1}}\setminus Q_{2^n})$ we have $\max_j a_j>2^n$ and hence
$\min_j a_j\ge K^{-1}\max_j a_j>2^n/K$, so
\[
  \prod_{j=1}^d a_j^{-\sigma}\le \Bigl(\frac{2^n}{K}\Bigr)^{-d\sigma}.
\]
Therefore,
\[
  \zeta_S(\boldsymbol{\sigma})
  \le
  |S\cap Q_1|+\sum_{n\ge0} |S\cap Q_{2^{n+1}}|\cdot \Bigl(\frac{2^n}{K}\Bigr)^{-d\sigma}
  \ll_{d,\alpha,K,\sigma}
  \sum_{n\ge0} 2^{nd(1/\alpha-\sigma)}<\infty,
\]
since $\sigma>1/\alpha$. Hence $\boldsymbol{\sigma}\in\mathcal D_S$ and
\[
  \Lambda_S\le \sigma_1+\cdots+\sigma_d = d\sigma = \frac{d}{\alpha}+d\varepsilon.
\]
Because $S$ is uniformly $K$--balanced, Theorem~\ref{thm:zeta-vs-singular}(4) gives
$s_\ast=\tfrac12\Lambda_S$, so
\[
  s_\ast \le \frac{d}{2\alpha}+\frac{d\varepsilon}{2}.
\]
Finally, by Theorem~\ref{thm:zeta-vs-singular}(2), we have
\[
  \dim_H(\mathcal E_S^{\mathrm{vec}})
  \le
  \dim_H(\mathcal E_S)
  \le
  s_\ast
  \le
  \frac{d}{2\alpha}+\frac{d\varepsilon}{2}.
\]
Letting $\varepsilon\downarrow 0$ yields $\dim_H(\mathcal E_S^{\mathrm{vec}})
  \le \dim_H(\mathcal E_S)\le \frac{d}{2\alpha}$.
\end{proof}

\subsection{Universal upper bounds via singular values and Dirichlet series}
\label{subsec:finite-modifications}

In this subsection we develop a collection of general tools for obtaining
upper bounds on the Hausdorff dimension of digit--restricted sets.
We show that finite modifications of digit constraints do not affect dimension,
establish a universal upper bound in terms of a singular--value type potential,
and relate this geometric quantity to a multivariate Dirichlet series
invariant.

\begin{lemma}[Digit insertion is bi-Lipschitz]
\label{lem:digit-insertion}
Let $X\subset(0,1)^d$ and fix digit vectors $\mathbf{a}_1,\dots,\mathbf{a}_N\in\N^d$.
Define $\Phi:X\to(0,1)^d$ by inserting $\mathbf{a}_1,\dots,\mathbf{a}_N$ as the first
$N$ digits in each coordinate, i.e.\ for $\mathbf{y}=(y_1,\dots,y_d)\in X$,
\[
x_j
=
\frac{1}{
a_{1,j}
+\displaystyle\frac{1}{a_{2,j}
+\cdots
+\displaystyle\frac{1}{a_{N,j}+y_j}}
},
\qquad 1\le j\le d.
\]
Let $X_{\mathbf{a}_1,\dots,\mathbf{a}_N}:=\Phi(X)$. Then
\[
\dim_H(X_{\mathbf{a}_1,\dots,\mathbf{a}_N})=\dim_H(X).
\]
\end{lemma}

\begin{proof}
For each coordinate, the map $y_j\mapsto x_j$ is a Möbius transformation
\[
x_j=\frac{P_{N,j}+y_j P_{N-1,j}}{Q_{N,j}+y_j Q_{N-1,j}},
\]
with convergents determined by $a_{1,j},\dots,a_{N,j}$. On $y_j\in(0,1)$ the derivative
is bounded above and below by positive constants depending only on the prescribed digits.
Hence each coordinate map is bi-Lipschitz on $(0,1)$ with uniform constants, and the
product map $\Phi$ is bi-Lipschitz on $(0,1)^d$. Hausdorff dimension is invariant under
bi-Lipschitz maps.
\end{proof}

\begin{lemma}[Deleting finitely many digits does not change $\dim_H$]
\label{lem:delete-finite}
Fix $S\subset\N^d$ and $N\in\N$. Then
\[
\dim_H(\mathcal{E}_S)
=
\dim_H\Bigl\{
\mathbf{x}:\ \mathbf{a}_n(\mathbf{x})\in S\ (n\ge N),\ a_n(x_j)\to\infty\ (1\le j\le d)
\Bigr\}.
\]
The same statement holds with $\mathcal{E}_S$ replaced by $\widetilde{\mathcal{E}}_S$.
\end{lemma}

\begin{proof}
Apply Lemma~\ref{lem:digit-insertion} to insert/delete the first $N-1$ digit vectors.
The relevant set is a countable union over all choices of the first $N-1$ digits, and each branch
is bi-Lipschitz equivalent to the tail set.
\end{proof}

For $\mathbf{a}=(a_1,\dots,a_d)\in\N^d$ let $a_{(1)}\le\cdots\le a_{(d)}$ be the nondecreasing
rearrangement of its coordinates. For $s\in(0,d]$ let $k$ be the unique integer with $k<s\le k+1$ and define
\begin{equation}\label{eq:def-phis}
\phi^s(\mathbf a)
:=
\Bigl(\prod_{j=1}^k a_{(j)}^{-2}\Bigr)\cdot a_{(k+1)}^{-2(s-k)}.
\end{equation}

For a family $(S_n)_{n\ge1}$ with $S_n\subset\N^d$ define the corresponding sets
\[
\mathcal{E}_{(S_n)}
:=
\Bigl\{
\mathbf{x}\in((0,1)\setminus\Q)^d:
\mathbf{a}_n(\mathbf{x})\in S_n\ (n\ge1),\ \ a_n(x_j)\to\infty\ (1\le j\le d)
\Bigr\},
\]
and similarly $\widetilde{\mathcal{E}}_{(S_n)}$ without the divergence requirement.

\begin{proposition}[Universal upper bound]
\label{prop:upper-bound-universal}
Let $(S_n)_{n\ge1}$ be subsets of $\N^d$ and fix $s\in(0,d]$. Set
\[
\Phi_{S_n}(s):=\sum_{\mathbf a\in S_n}\phi^s(\mathbf a).
\]
If
\[
\limsup_{n\to\infty}\Bigl(\prod_{k=1}^n\Phi_{S_k}(s)\Bigr)^{1/n}<1,
\]
then $\dim_H(\widetilde{\mathcal{E}}_{(S_n)})\le s$, and hence also
$\dim_H(\mathcal{E}_{(S_n)})\le s$.
In particular, for fixed $S$ one has $\dim_H(\widetilde{\mathcal{E}}_S)\le s$ whenever $\Phi_S(s)<1$.
\end{proposition}

\begin{proof}
Fix $s\in(0,d]$ and let $k\in\{0,1,\dots,d-1\}$ be the unique integer such that
$k<s\le k+1$.
Assume that
\[
\limsup_{n\to\infty}\Bigl(\prod_{j=1}^n\Phi_{S_j}(s)\Bigr)^{1/n}<1.
\]
Then there exist $\rho\in(0,1)$ and $n_0\in\N$ such that
\begin{equation}\label{eq:Phi-decay}
\prod_{j=1}^n\Phi_{S_j}(s)\le \rho^n
\qquad\text{for all }n\ge n_0.
\end{equation}

For each $n\ge1$ we have the cylinder cover
\begin{equation}\label{eq:cylinder-cover}
\widetilde{\mathcal E}_{(S_n)}
\subset
\bigcup_{\omega_n\in S_1\times\cdots\times S_n} I_{\omega_n},
\end{equation}
where $\omega_n=(\mathbf a_1,\dots,\mathbf a_n)$ and
\[
I_{\omega_n}=\prod_{j=1}^d I(a_{1,j},\dots,a_{n,j})
\]
is the product cylinder.  Write $L_j(\omega_n)$ for the side length of
$I_{\omega_n}$ in coordinate $j$.
By Lemma~\ref{lem:cylinder_structure} and Lemma~\ref{lem:1d_geometry}(i), there
exist constants $c_0,C_0>0$ such that for every $1\le j\le d$,
\begin{equation}\label{eq:side-length}
c_0\prod_{i=1}^n (a_{i,j}+1)^{-2}
\ \le\
L_j(\omega_n)
\ \le\
C_0\prod_{i=1}^n a_{i,j}^{-2}.
\end{equation}
Let $\mathcal L_1(\omega_n)\ge\cdots\ge\mathcal L_d(\omega_n)$ be the
nonincreasing rearrangement of the side lengths $\{L_j(\omega_n)\}_{j=1}^d$ and
set
\begin{equation}\label{eq:varphi-cylinder}
\varphi^s(I_{\omega_n})
:=
\mathcal L_1(\omega_n)\cdots \mathcal L_k(\omega_n)\,
\mathcal L_{k+1}(\omega_n)^{\,s-k}.
\end{equation}

Recall that the $s$--dimensional Hausdorff content of a set $E$ is
\[
\mathcal H^s_\infty(E)
=
\inf\Bigl\{
\sum_{m} \diam(U_m)^s:\ E\subset\bigcup_m U_m
\Bigr\}.
\]
A standard covering argument for axis-parallel rectangles yields a constant
$C_1=C_1(d,s)>0$ and, for every $\omega_n$, a cover
$\mathscr U(\omega_n)$ of $I_{\omega_n}$ such that
\begin{equation}\label{eq:content-rect}
\sum_{U\in\mathscr U(\omega_n)}\diam(U)^s
\le
C_1\,\varphi^s(I_{\omega_n}),
\qquad
\max_{U\in\mathscr U(\omega_n)}\diam(U)
\le C_d\diam(I_{\omega_n}).
\end{equation}
Indeed, one covers the rectangle by cubes whose side length is comparable to
$\mathcal L_{k+1}(\omega_n)$.

Using \eqref{eq:side-length}, the submultiplicativity of the singular-value
function for diagonal maps, and the definition \eqref{eq:def-phis}, there is a
constant $C_2=C_2(d,s)>0$ such that for every word
$\omega_n=(\mathbf a_1,\dots,\mathbf a_n)$,
\begin{equation}\label{eq:varphi-phi}
\varphi^s(I_{\omega_n})
\le
C_2\prod_{i=1}^n \phi^s(\mathbf a_i).
\end{equation}
Define
\[
\mathcal W_n(s)
:=
\sum_{\omega_n\in S_1\times\cdots\times S_n}\varphi^s(I_{\omega_n}).
\]
Then \eqref{eq:varphi-phi} gives
\[
\mathcal W_n(s)
\le
C_2\sum_{\omega_n\in S_1\times\cdots\times S_n}\ \prod_{i=1}^n \phi^s(\mathbf a_i)
=
C_2\prod_{i=1}^n\Bigl(\sum_{\mathbf a\in S_i}\phi^s(\mathbf a)\Bigr)
=
C_2\prod_{i=1}^n \Phi_{S_i}(s).
\]
Combining these covers over all words of length $n$ gives
\[
\sum_{\omega_n\in S_1\times\cdots\times S_n}
\ \sum_{U\in\mathscr U(\omega_n)}\diam(U)^s
\le
C_1\,\mathcal W_n(s)
\le
C_1C_2\prod_{i=1}^n \Phi_{S_i}(s).
\]
By Remark~\ref{rmk:uniform-cylinder-decay-d}, the mesh of this cover tends to zero
uniformly as $n\to\infty$.  Together with \eqref{eq:Phi-decay}, the preceding bound
therefore implies
$\mathcal H^s(\widetilde{\mathcal E}_{(S_n)})=0$, and hence
$\dim_H(\widetilde{\mathcal E}_{(S_n)})\le s$.
Since $\mathcal E_{(S_n)}\subset\widetilde{\mathcal E}_{(S_n)}$, the same upper
bound holds for $\mathcal E_{(S_n)}$.

In the stationary case $S_n\equiv S$, the hypothesis reduces to $\Phi_S(s)<1$,
and the conclusion follows immediately.
\end{proof}

Define
\[
\zeta_S(\boldsymbol{\sigma})
:=
\sum_{\mathbf a\in S}\prod_{j=1}^d a_j^{-\sigma_j},
\qquad
\mathcal D_S
:=
\{\boldsymbol{\sigma}\in(0,\infty)^d:\ \zeta_S(\boldsymbol{\sigma})<\infty\},
\]
and the trace infimum
\[
\Lambda_S
:=
\inf_{\boldsymbol{\sigma}\in\mathcal D_S}(\sigma_1+\cdots+\sigma_d).
\]
We will need a small combinatorial device to compare weighted sums after sorting
coordinates; for completeness we record the following rearrangement inequality.

\begin{lemma}[Rearrangement inequality]\label{lem:rearrangement}
Let $x_1\le\cdots\le x_d$ and $w_1\ge\cdots\ge w_d$ be real numbers. Then for every
permutation $\pi\in S_d$,
\begin{equation}\label{eq:rearrangement}
\sum_{i=1}^d w_{\pi(i)}x_i \ \ge\ \sum_{i=1}^d w_i x_i .
\end{equation}
\end{lemma}

\begin{proof}
If $p<q$ and $w_{\pi(p)}<w_{\pi(q)}$, let $\pi'$ be obtained from $\pi$ by swapping
$\pi(p)$ and $\pi(q)$. Then
\[
\sum_{i=1}^d w_{\pi(i)}x_i-\sum_{i=1}^d w_{\pi'(i)}x_i
=
(w_{\pi(q)}-w_{\pi(p)})(x_q-x_p)\ge0.
\]
Iterating this swap removes inversions and yields \eqref{eq:rearrangement}.
\end{proof}

\begin{proposition}[Zeta trace controls $\Phi_S$]
\label{prop:zeta-vs-singular}
Let $S\subset\N^d$ be infinite and define
\[
s_\ast:=\inf\{0<s\le d:\ \Phi_S(s)<\infty\}.
\]
Then
\[
\Lambda_S\le 2s_\ast .
\]
\end{proposition}

\begin{proof}
Fix $s\in(0,d]$ such that
$\Phi_S(s)=\sum_{\mathbf a\in S}\phi^s(\mathbf a)<\infty$.
Let $k\in\{0,1,\dots,d-1\}$ be the unique integer such that $k<s\le k+1$.

Define the exponent vector
\[
\widetilde{\boldsymbol{\sigma}}(s)
:=
(\underbrace{2,\dots,2}_{k\ \text{times}},\ 2(s-k),\ \underbrace{0,\dots,0}_{d-k-1\ \text{times}})
\in[0,\infty)^d,
\]
whose coordinates are non-increasing.
For $\mathbf a=(a_1,\dots,a_d)\in\N^d$, let $a_{(1)}\le\cdots\le a_{(d)}$ be the
nondecreasing rearrangement of its coordinates and set $x_i:=\log a_{(i)}$, so
$x_1\le\cdots\le x_d$.

Let $\tau\in S_d$ be the permutation such that $a_{\tau(i)}=a_{(i)}$ for
$1\le i\le d$. Applying Lemma~\ref{lem:rearrangement} with the nondecreasing
sequence $(x_i)_{i=1}^d$ and the nonincreasing weights
$(\widetilde{\sigma}_i(s))_{i=1}^d$, and taking $\pi=\tau$, we obtain
\[
\sum_{i=1}^d \widetilde{\sigma}_{\tau(i)}(s)\,x_i
\ \ge\
\sum_{i=1}^d \widetilde{\sigma}_i(s)\,x_i
=
2\sum_{i=1}^k x_i + 2(s-k)x_{k+1}.
\]
Since $\sum_{i=1}^d \widetilde{\sigma}_{\tau(i)}(s)\,x_i
=\sum_{j=1}^d \widetilde{\sigma}_j(s)\,\log a_j$, this yields
\[
\sum_{j=1}^d \widetilde{\sigma}_j(s)\,\log a_j
\ \ge\
2\sum_{i=1}^k \log a_{(i)} + 2(s-k)\log a_{(k+1)}.
\]
Exponentiating and using the definition of $\phi^s$ gives the pointwise bound
\[
\prod_{j=1}^d a_j^{-\widetilde{\sigma}_j(s)}
\le
\Bigl(\prod_{i=1}^k a_{(i)}^{-2}\Bigr)\,a_{(k+1)}^{-2(s-k)}
=
\phi^s(\mathbf a).
\]
For $\eta>0$ define
\[
\boldsymbol{\sigma}_\eta(s)
:=\widetilde{\boldsymbol{\sigma}}(s)+\eta(1,\dots,1)\in(0,\infty)^d.
\]
Since every $a_j\ge1$, the preceding pointwise bound remains valid after adding
$\eta$ to every exponent. Summing over $\mathbf a\in S$ gives
\[
\zeta_S\bigl(\boldsymbol{\sigma}_\eta(s)\bigr)
\le
\sum_{\mathbf a\in S}\phi^s(\mathbf a)
=
\Phi_S(s)<\infty.
\]
Thus $\boldsymbol{\sigma}_\eta(s)\in\mathcal D_S$ and
\[
\Lambda_S\le 2s+d\eta.
\]
Letting $\eta\downarrow0$ gives $\Lambda_S\le2s$. Taking the infimum over the
admissible $s\in(0,d]$ completes the proof.
\end{proof}

\begin{remark}\label{rem:zeta-fibers}
The convergence region $\mathcal D_S$ is a global summability invariant and may
be dominated by a single thick coordinate fiber.
In particular, even if $S$ is sparse in most directions, the presence of one
coordinate fiber with slow decay can force $\Lambda_S$ to be large.
This reflects the fact that $\zeta_S$ records worst--case behavior across
coordinates rather than any frequency or distributional information.

If $S$ is uniformly $K$--balanced, then the above pathology disappears.
Indeed, for $\sigma>0$ one has
\[
\prod_{j=1}^d a_j^{-\sigma}\asymp_K a_{(1)}^{-\sigma d},
\qquad
\phi^s(\mathbf a)\asymp_K\prod_{j=1}^d a_j^{-2s/d}.
\]
Consequently, the diagonal Dirichlet series
$\zeta_S(\sigma,\dots,\sigma)$ and the scalar series $\Phi_S(s)$ become equivalent
after the change of variables $\sigma=2s/d$, and in this case
\[
s_\ast=\tfrac12\,\Lambda_S .
\]
\end{remark}

\subsection{Proof of Theorem~\ref{thm:zeta-vs-singular}}

\begin{proof}[Proof of Theorem~\ref{thm:zeta-vs-singular}]
(1) By Proposition~\ref{prop:upper-bound-universal} with $S_n\equiv S$, if
$\Phi_S(s)<1$ then $\dim_H(\widetilde{\mathcal E}_S)\le s$.
Taking the infimum over such $s$ yields
$\dim_H(\widetilde{\mathcal E}_S)\le s_\sharp$.

(2) If $s_\ast=d$, the conclusion $\dim_H(\mathcal E_S)\le d$ is immediate.
Assume $s_\ast<d$ and fix $s\in(s_\ast,d]$. Then
$\Phi_S(s)=\sum_{\mathbf a\in S}\phi^s(\mathbf a)<\infty$.
For $M\in\N$ set
\[
S^{\ge M}:=\{\mathbf a\in S:\ a_{(1)}(\mathbf a)=\min_j a_j \ge M\}.
\]
Since $(S^{\ge M})_{M\ge1}$ is a decreasing family with $\bigcap_{M\ge1}S^{\ge M}=\varnothing$,
monotone convergence gives $\Phi_{S^{\ge M}}(s)\downarrow 0$ as $M\to\infty$.
Hence we may choose $M$ such that
\begin{equation}\label{eq:choose-M-Phi-lt1}
\Phi_{S^{\ge M}}(s)<1.
\end{equation}

Now let $\mathcal E_S^{(M,N)}$ be the set of points $\mathbf x\in((0,1)\setminus\Q)^d$ such that
$\mathbf a_n(\mathbf x)\in S$ for $1\le n<N$, $\mathbf a_n(\mathbf x)\in S^{\ge M}$ for all $n\ge N$,
and $a_n(x_j)\to\infty$ for each $1\le j\le d$.
Since $\min_j a_n(x_j)\to\infty$ for $\mathbf x\in\mathcal E_S$, every $\mathbf x\in\mathcal E_S$
belongs to $\mathcal E_S^{(M,N)}$ for some $N$, and therefore
\[
\mathcal E_S \subset \bigcup_{N\ge1} \mathcal E_S^{(M,N)}.
\]

By Lemma~\ref{lem:delete-finite}, deleting (or inserting) finitely many initial digit vectors does not
change Hausdorff dimension, hence for each $N$,
\[
\dim_H(\mathcal E_S^{(M,N)})=\dim_H(\mathcal E_{S^{\ge M}}).
\]
Consequently, since the Hausdorff dimension of a countable union is the supremum of the
Hausdorff dimensions of the pieces, we obtain
\[
\dim_H(\mathcal E_S)
\leq 
\dim_H\Bigl(\bigcup_{N\ge1}\mathcal E_S^{(M,N)}\Bigr)
\le
\sup_{N\ge1}\dim_H(\mathcal E_S^{(M,N)})
=
\dim_H(\mathcal E_{S^{\ge M}}).
\]

Finally, applying Proposition~\ref{prop:upper-bound-universal} to the stationary family
$S_n\equiv S^{\ge M}$ and using \eqref{eq:choose-M-Phi-lt1}, we obtain
$\dim_H(\mathcal E_{S^{\ge M}})\le s$, and therefore $\dim_H(\mathcal E_S)\le s$.
Since $s>s_\ast$ was arbitrary, it follows that $\dim_H(\mathcal E_S)\le s_\ast$.
In particular,
\[
  \dim_H(\mathcal E_S^{\mathrm{vec}})
  \le
  \dim_H(\mathcal E_S)
  \le
  s_\ast.
\]

(3) This is Proposition~\ref{prop:zeta-vs-singular}.

(4) Suppose that $S$ is uniformly $K$--balanced, that is,
$\max_j a_j \le K \min_j a_j$ for all $\mathbf a\in S$.
Then for every $\sigma>0$ and $\mathbf a\in S$,
\[
\prod_{j=1}^d a_j^{-\sigma}
\asymp_K
a_{(1)}^{-\sigma d},
\]
where $a_{(1)}=\min_j a_j$.
Similarly, for $s\in(0,d]$ one has
\[
\phi^s(\mathbf a)
=
\Bigl(\prod_{j=1}^k a_{(j)}^{-2}\Bigr)\,a_{(k+1)}^{-2(s-k)}
\asymp_K
\prod_{j=1}^d a_j^{-2s/d},
\]
with implicit constants depending only on $K,d,s$.

Consequently, after the change of variables $\sigma=2s/d$, the diagonal Dirichlet
series and the singular series are equivalent. In particular, $\Lambda_S\le2s_\ast$.
For the reverse inequality, fix
$\boldsymbol\sigma=(\sigma_1,\dots,\sigma_d)\in\mathcal D_S$ and put
$T:=\sigma_1+\cdots+\sigma_d$. If $T\le2d$, uniform balance gives
\[
\phi^{T/2}(\mathbf a)
\asymp_{K,d,\boldsymbol\sigma}
\prod_{j=1}^d a_j^{-\sigma_j}
\qquad(\mathbf a\in S),
\]
so $\Phi_S(T/2)<\infty$ and hence $2s_\ast\le T$. If $T>2d$, then
$\Phi_S(d)\le\sum_{\mathbf a\in\N^d}\prod_{j=1}^d a_j^{-2}<\infty$, and again
$2s_\ast\le2d<T$. Taking the infimum over
$\boldsymbol\sigma\in\mathcal D_S$ yields $2s_\ast\le\Lambda_S$.
Therefore $\Lambda_S=2s_\ast$, as claimed.

\end{proof}

\subsection{Proof of Theorem~\ref{thm:full-restriction}}

A key step is to bound the Hausdorff dimension of certain product sets arising from
one--dimensional tail continued--fraction sets. We reduce the $d$--dimensional
coordinatewise problem to one dimension using two standard inputs: (i) for tail
alphabets in dimension one, $\dim_H=\dim_B=\dim_P$; and (ii) a product inequality for
Hausdorff dimension whose upper bound involves packing dimension. We will also use
a quantitative estimate for $\dim_H(E_{\ge N})$ due to Good.

For $N\ge2$ set
\[
E_{\ge N}:=\{x\in(0,1)\setminus\Q: a_n(x)\ge N \text{ for all } n\ge1\}.
\]
\begin{lemma}\label{lem:HPB-tail}
We have
\[
\dim_H(E_{\ge N})=\dim_B(E_{\ge N})=\dim_P(E_{\ge N}),
\]
where $\dim_B$ and $\dim_P$ denote the (upper) box and packing dimensions.
\end{lemma}

\begin{proof}
This is a standard fact for one--dimensional continued fractions with a tail alphabet.
A direct reference is \cite[Section~6, Remark, p.~5023]{MauldinUrbanski1999}, which
records the coincidence of Hausdorff, box, and packing dimensions in this setting.

Alternatively, one may combine \cite[Corollary~5.9]{MauldinUrbanski1999} with the
general identity $\dim_P=\dim_B$ for conformal iterated function systems
(\cite[Theorem~3.1]{MauldinUrbanski1996_PLMS}) and the discussion in
\cite[Section~6]{MauldinUrbanski1999}.
\end{proof}

\begin{remark}
The coincidence in Lemma~\ref{lem:HPB-tail} is special to the tail alphabet and is
not asserted for general infinite continued-fraction alphabets; compare
\cite[Theorem~6.2]{MauldinUrbanski1999}.
\end{remark}

We will use the following product inequality; see, e.g., \cite{Marstrand1954}.

\begin{lemma}\label{lem:marstrand-product}
Suppose that $A\subset\R^{d}$ and $B\subset\R^{t}$ are Borel sets. Then
\begin{equation}\label{eq:marstrand-product}
\dim_H A+\dim_H B
\ \le\
\dim_H(A\times B)
\ \le\
\dim_H A+\dim_P B,
\end{equation}
where $\dim_P$ denotes the packing dimension.
\end{lemma}

\begin{lemma}[Good's estimate for large partial quotients {\cite{Good1941}}]
\label{lem:good1941}
Assume $N\ge 20$. Then
\[
\frac12+\frac{1}{2\log(N+2)}
\;<\;
\dim_H(E_{\ge N})
\;<\;
\frac12+\frac{\log\log(N-1)}{2\log(N-1)}.
\]
\end{lemma}

\begin{proof}[Proof of Theorem~\ref{thm:full-restriction}]
For $N\ge2$ set
\[
E_{\ge N}:=\{x\in(0,1)\setminus\Q: a_n(x)\ge N \text{ for all }n\ge1\},
\qquad
\widetilde{\mathcal E}_N=(E_{\ge N})^d.
\]
For $N\ge20$, Lemma~\ref{lem:good1941} gives
$\dim_H(E_{\ge N})>\tfrac12$. Iterating the left inequality in
Lemma~\ref{lem:marstrand-product}, we obtain
\begin{equation}\label{eq:lower-from-Einfty}
\dim_H\bigl((E_{\ge N})^d\bigr)
\ge d\,\dim_H(E_{\ge N})
>\frac d2.
\end{equation}

By Lemma~\ref{lem:HPB-tail} we have
\begin{equation}\label{eq:HPB-equal}
\dim_H(E_{\ge N})=\dim_B(E_{\ge N})=\dim_P(E_{\ge N}).
\end{equation}

We first obtain an upper bound for $\dim_H(\widetilde{\mathcal E}_N)$ in terms of
$\dim_H(E_{\ge N})$. By Lemma~\ref{lem:marstrand-product} and \eqref{eq:HPB-equal},
\[
\dim_H\bigl((E_{\ge N})\times (E_{\ge N})\bigr)
\le
\dim_H(E_{\ge N})+\dim_P(E_{\ge N})
=
2\,\dim_H(E_{\ge N}).
\]
Iterating this estimate gives
\begin{equation}\label{eq:product-upper}
\dim_H(\widetilde{\mathcal E}_N)
=
\dim_H\bigl((E_{\ge N})^d\bigr)
\le
d\,\dim_H(E_{\ge N}).
\end{equation}

On the other hand, Lemma~\ref{lem:good1941} implies that $\dim_H(E_{\ge N})\to\tfrac12$
as $N\to\infty$. Combining this with \eqref{eq:product-upper} yields
\[
\limsup_{N\to\infty}\dim_H(\widetilde{\mathcal E}_N)
\le
\limsup_{N\to\infty} d\,\dim_H(E_{\ge N})
=
\frac d2.
\]
Together with \eqref{eq:lower-from-Einfty}, we conclude that
\[
\lim_{N\to\infty}\dim_H(\widetilde{\mathcal E}_N)=\frac d2.
\]
The stated $\varepsilon$--form follows immediately.
\end{proof}

\section{Moran Fractal Constructions in Higher Dimensions}\label{sec:moran}   
  
This section recalls a standard lower bound for Moran--type constructions in $\R^d$
using the mass distribution principle; see \cite{Falconer2014,Mattila1995} for background.
The resulting bounds need not be sharp in general; the applications below use compact
continued--fraction cylinder rectangles.

    \begin{definition}[Moran fractal]\label{def:moran-nd}
    Let $d\in\mathbb{N}$. Let $(r_n)_{n\ge1}$ and $(\delta_n)_{n\ge1}$ satisfy
    $r_n\ge2$ and $0<\delta_{n+1}<\delta_n$ for all $n$.
    A compact set $F\subset[0,1]^d$ is called a \emph{Moran fractal with parameters
    $(r_n)$ and $(\delta_n)$} if $\mathcal Q_0:=\{[0,1]^d\}$ and there exist families
    $\mathcal{Q}_n$ ($n\ge1$) of finitely many pairwise disjoint nonempty compact
    subsets of $[0,1]^d$ such that, writing
    \[
    E_n:=\bigcup_{Q\in\mathcal{Q}_n}Q,\qquad F:=\bigcap_{n\ge0}E_n,
    \]
    the following hold for every $n\ge1$:
    \begin{enumerate}
    \item \textup{(Nestedness)} For each $Q\in\mathcal{Q}_n$ there is a unique
          parent $\widehat Q\in\mathcal{Q}_{n-1}$ with $Q\subset \widehat Q$.
    \item \textup{(Branching)} Each parent basic set $P\in\mathcal{Q}_{n-1}$ contains at
          least $r_n$ distinct children from $\mathcal{Q}_n$.
    \item \textup{(Separation inside a parent)} If $Q,Q'\in\mathcal{Q}_n$ are distinct
          and have the same parent in $\mathcal{Q}_{n-1}$, then
          $\operatorname{dist}(Q,Q')\ge \delta_n$.
    \item \textup{(Vanishing mesh)} $\max_{Q\in\mathcal{Q}_n}\diam(Q)\to 0$ as $n\to\infty$.
    \end{enumerate}
    \end{definition}
    
    \begin{lemma}[Generalized Moran's formula, lower bound]\label{lem:moran-lb-nd}
    Let $F\subset [0,1]^d$ be a Moran fractal with parameters $(r_n)$ and $(\delta_n)$.
    Then
    \begin{equation}\label{eq:moran-lower-bound}
    \dim_H F
    \;\ge\;
    \liminf_{n\to\infty}
    \frac{\log(r_1\cdots r_{n-1})}{-\log(r_n^{1/d}\delta_n)}.
    \end{equation}
    \end{lemma}

    The proof of Lemma~\ref{lem:moran-lb-nd} relies on several lemmas. We state them first.
    
    \begin{lemma}[Packing estimate in $\R^d$]\label{lem:packing}
    Let $d\in\N$. There exists a constant $C_d\ge1$ depending only on $d$ such that
    the following holds.
    
    Let $A\subset\R^d$ be bounded and let $\delta>0$. Suppose $S\subset A$ is finite and
    \[
    \|x-y\|\ge \delta\qquad\text{for all distinct }x,y\in S.
    \]
    Then
    \[
    \#S \le C_d\Bigl(\frac{\diam(A)}{\delta}+1\Bigr)^d.
    \]
    In particular, if $\#S\ge2$, then
    \[
    \diam(A)\ \ge\ c_d\,(\#S)^{1/d}\,\delta
    \]
    for some constant $c_d>0$ depending only on $d$.
    \end{lemma}
    
    \begin{proof}
    Write $m:=\#S$ and $D:=\diam(A)$. Fix $x_0\in A$. Then $A\subset B(x_0,D)$.
    For each $x\in S$ consider the closed ball $B(x,\delta/2)$. These balls are pairwise
    disjoint since $\|x-y\|\ge\delta$ for $x\neq y$. Moreover, for any $y\in B(x,\delta/2)$,
    \[
    \|y-x_0\|\le \|y-x\|+\|x-x_0\|\le \frac{\delta}{2}+D,
    \]
    so $B(x,\delta/2)\subset B(x_0,D+\delta/2)$.
    
    Let $\operatorname{vol}\bigl(B(0,1)\bigr)$ denote the volume of the unit ball in $\mathbb{R}^d$.
    Comparing volumes gives
    \[
    m\cdot \operatorname{vol}\bigl(B(0,1)\bigr)\Bigl(\frac{\delta}{2}\Bigr)^d
    \le \operatorname{vol}\Bigl(B\Bigl(x_0,D+\frac{\delta}{2}\Bigr)\Bigr)
    = \operatorname{vol}\bigl(B(0,1)\bigr)\Bigl(D+\frac{\delta}{2}\Bigr)^d.
    \]
    
    Cancelling $v_d$ yields
    \[
    m \le \Bigl(\frac{D+\delta/2}{\delta/2}\Bigr)^d
    = \Bigl(\frac{2D}{\delta}+1\Bigr)^d
    \le 3^d\Bigl(\frac{D}{\delta}+1\Bigr)^d,
    \]
    which proves the first claim with $C_d:=3^d$.
    
    For the second claim, assume $m\ge2$. Then $D\ge \delta$ (since $S$ contains two
    points at distance at least $\delta$), hence $2D/\delta+1\le 3D/\delta$. Therefore
    \[
    m \le \Bigl(\frac{2D}{\delta}+1\Bigr)^d \le \Bigl(\frac{3D}{\delta}\Bigr)^d,
    \]
    so $D \ge 3^{-1} m^{1/d}\delta$. Thus we may take $c_d:=1/3$.
    \end{proof}

    We now state the second lemma, which establishes an upper bound for the lower limit.
    \begin{lemma}\label{lem:s-at-most-d}
    With the notation of Lemma~\ref{lem:moran-lb-nd}, assume moreover that
    $r_k \ge 2$ for all $k\in\mathbb{N}$.
    Then
    \[
    \liminf_{n\to\infty}
    \frac{\log(r_1\cdots r_{n-1})}{-\log L_n} \;\le\; d,
    \]
    where $L_n := r_n^{1/d}\delta_n$.  
    In particular, $L_n \to 0$ as $n\to\infty$, and therefore the right-hand side of
    \eqref{eq:moran-lower-bound} in Lemma~\ref{lem:moran-lb-nd} is well defined
    (for all sufficiently large $n$ one has $0 < L_n < 1$ and $-\log L_n > 0$).
    \end{lemma}

    \begin{proof}
    Write $N_n:=\#\mathcal{Q}_n$. Since each $(k-1)$-level basic set has at least $r_k$ children,
    \[
    N_n\ge r_1\cdots r_n\qquad(n\ge1).
    \]
    
    We claim that $\mathcal{Q}_n$ is $\delta_n$--separated. Take $Q\neq Q'\in\mathcal{Q}_n$.
    For $0\le k\le n$ let $Q^{(k)}$ and $(Q')^{(k)}$ be the unique $k$-level ancestors.
  Let $m$ be the least index with $Q^{(m)}\neq (Q')^{(m)}$. Then
  $Q^{(m-1)}=(Q')^{(m-1)}$, so $Q^{(m)}$ and $(Q')^{(m)}$ are two distinct
  children of the same parent. Hence
  \[
  \dist\bigl(Q^{(m)},(Q')^{(m)}\bigr)\ge \delta_m.
  \]
  Since $Q\subset Q^{(m)}$ and $Q'\subset (Q')^{(m)}$, we have
  \[
  \dist(Q,Q')\ge \dist\bigl(Q^{(m)},(Q')^{(m)}\bigr)\ge \delta_m.
  \]
  Also $m\le n$ and $(\delta_k)$ is decreasing, so $\delta_m\ge \delta_n$.
  Therefore $\dist(Q,Q')\ge \delta_n$.

    Pick $x_Q\in Q$ for each $Q\in\mathcal{Q}_n$. Then $\{x_Q\}$ is $\delta_n$--separated
    in $[0,1]^d$. By Lemma~\ref{lem:packing},
    \[
    N_n\le C\,\delta_n^{-d}
    \]
    for some $C>0$ depending only on $d$. Thus
    \[
    r_1\cdots r_{n-1}\le \frac{C}{r_n\delta_n^d}=\frac{C}{L_n^d}.
    \]
    If $r_k\ge2$ for all $k$, then $r_1\cdots r_{n-1}\ge 2^{n-1}$, so
    $L_n^d\le C2^{-(n-1)}$ and $L_n\to0$. Hence $-\log L_n\to\infty$ and, for large $n$,
    \[
    \frac{\log(r_1\cdots r_{n-1})}{-\log L_n}
    \le d+\frac{\log C}{-\log L_n}.
    \]
    Therefore
    \[
    \liminf_{n\to\infty}\frac{\log(r_1\cdots r_{n-1})}{-\log L_n}\le d.
    \]
    \end{proof}
    \begin{lemma}[Mass distribution principle]\label{lem:mdp}
    Let $F\subset \R^d$ be a Borel set. Assume there exist $s\ge0$, constants $C>0$ and
    $r_0>0$, and a Borel probability measure $\mu$ supported on $F$ such that
    \[
    \mu\bigl(B(x,r)\bigr)\le C\,r^s
    \qquad\text{for all }x\in\R^d,\ 0<r\le r_0.
    \]
    Then $\dim_H(F)\ge s$.
    \end{lemma}
    
    \begin{proof}
    This is the standard mass distribution principle; see \cite[Proposition~4.2]{Falconer2014}.
    \end{proof}

    Now we give Proof of Lemma~\ref{lem:moran-lb-nd}:

   \begin{proof}
Write $\mathcal{Q}_n$ for the family of $n$-level basic sets, and set $R_n:=r_1\cdots r_n$.

If some parent has more than $r_n$ children, choose and fix exactly $r_n$ children
inside it. Do this for all $n$. This gives a Moran subset $F'\subset F$ with the
same parameters $(r_n),(\delta_n)$. Then $\dim_H F\ge \dim_H F'$. So we may assume
each parent has exactly $r_n$ children. Hence $\#\mathcal{Q}_n=R_n$.

Define a probability measure $\mu$ on $F$ by
\[
\mu(Q)=R_n^{-1}\qquad (Q\in\mathcal{Q}_n).
\]
This is consistent by nesting, so $\mu$ is well defined and supported on $F$.

Fix $0<r<\delta_2/2$ and set $U:=B(x,r)$. Then $\diam(U)=2r$.
Since $\delta_n\to0$ by Lemma~\ref{lem:s-at-most-d}, there is a unique $n\ge2$
such that $\delta_n\le 2r<\delta_{n-1}$.
Let $N(U)$ be the number of basic sets in $\mathcal{Q}_n$ that meet $U$.

We first note that $\mathcal{Q}_{n-1}$ is $\delta_{n-1}$--separated.
Indeed, take distinct $I,I'\in\mathcal{Q}_{n-1}$ and let $m$ be the first level
where their ancestors split. Then their $m$-level ancestors are siblings, so their
distance is at least $\delta_m\ge \delta_{n-1}$ since $(\delta_k)$ decreases.
Thus $\dist(I,I')\ge \delta_{n-1}$.

Hence $U$ meets at most one $(n\!-\!1)$-level cube (since $\diam(U)<\delta_{n-1}$).
So all $n$-level basic sets meeting $U$ lie in one parent, and therefore $N(U)\le r_n$.

Also pick one point $x_Q\in Q\cap U$ for each such cube $Q$.
Then the set $\{x_Q\}$ is $\delta_n$--separated in $U$, so by Lemma~\ref{lem:packing},
\[
N(U)\le C_d\Bigl(\frac{\diam(U)}{\delta_n}+1\Bigr)^d
\le C\,\Bigl(\frac{r}{\delta_n}\Bigr)^d
\]
for a constant $C$ depending only on $d$ (since $\delta_n\le 2r$).

Thus
\[
N(U)\le \min\Bigl\{r_n,\; C\Bigl(\frac{r}{\delta_n}\Bigr)^d\Bigr\}.
\]
Fix $s\in(0,d]$ and put $\theta:=s/d\in(0,1]$. Then
$\min\{A,B\}\le A^{1-\theta}B^\theta$, so
\[
N(U)\le r_n^{1-s/d}\cdot \Bigl(C\Bigl(\frac{r}{\delta_n}\Bigr)^d\Bigr)^{s/d}
= C^{s/d}\, r^s\, r_n^{1-s/d}\,\delta_n^{-s}.
\]
Since each $n$-level basic set has $\mu$-mass $R_n^{-1}$, we get
\[
\mu(U)\le R_n^{-1}N(U)
\le C^{s/d}\, r^s\cdot R_n^{-1} r_n^{1-s/d}\delta_n^{-s}
= C^{s/d}\, r^s\cdot \frac{1}{R_{n-1}(r_n^{1/d}\delta_n)^s}.
\]

Let
\[
\alpha:=\liminf_{n\to\infty}\frac{\log R_{n-1}}{-\log(r_n^{1/d}\delta_n)}.
\]
Take any $s<\alpha$ with $s\le d$. Then for all large $n$,
\[
R_{n-1}(r_n^{1/d}\delta_n)^s \ge 1.
\]
For such $r$ (hence such $n$) we have $\mu(B(x,r))\le C' r^s$, with $C'$ depending
only on $d$ and $s$. By Lemma~\ref{lem:mdp}, $\dim_H(F)\ge s$.

Since this holds for all $s<\alpha$ with $s\le d$, we obtain
\[
\dim_H(F)\ge \min\{d,\alpha\}.
\]
Finally, Lemma~\ref{lem:s-at-most-d} gives $\alpha\le d$, so $\min\{d,\alpha\}=\alpha$,
which is exactly \eqref{eq:moran-lower-bound}.
\end{proof}

\section{Three-stage construction in \texorpdfstring{$\N^d$}{Nd}}\label{sec:three-stage}

In this section we construct Moran--type seed sets in $\mathbb N^d$ and derive
lower bounds for their Hausdorff dimension in terms of the polynomial density
exponent of the underlying digit set.
The construction is intentionally flexible: no divergence or anisotropy
assumptions are imposed at this stage, and such conditions will be enforced
later only when required by the transference arguments.

The argument proceeds in three steps: we first extract a relatively thin subset
$S^\ast\subset S$ with controlled growth, then build an extreme seed Moran
fractal from $S^\ast$, and finally insert digits from a target set along
carefully chosen indices while preserving Hausdorff dimension via an
almost--Lipschitz elimination map.

\subsection{Construction of a relatively thin subset}
\label{subsec:step1}

We define a slowly increasing function $\nu: [1, \infty) \to \N$ by the factorial blocks:
\begin{equation}
	\nu(\xi) = k \quad \text{for } \xi \in [k!, (k+1)!), \quad k \in \N.
\end{equation}

\begin{lemma}[\cite{Nakajima2025}, Lemma 3.3]\label{lem:nu_bound}
For every $t > e$, there exists a $\lambda > 1$, depending on $t$, such that
\[
\nu(t^n) \leq (n \log t)^\lambda \quad \text{for all } n \in \mathbb{N}.
\]
\end{lemma}

We first show that from any infinite subset of $\N^d$ we can extract a subset of zero relative density with a slowly varying growth rate.

\begin{lemma}[Extraction of a thin subset in $\mathbb{N}^d$]\label{lem:thin-subset-Zd}
Let $S\subset\mathbb{N}^d$ be infinite. Then there exists $S^\ast\subset S$ such that:
\begin{itemize}[leftmargin=2em]
\item[(i)] $S^\ast$ has no ``nearest neighbour'' pairs in the sense that whenever $\mathbf{v},\mathbf{w}\in S^\ast$ with $\mathbf{v}\neq \mathbf{w}$, one has $\|\mathbf{v}-\mathbf{w}\|_\infty\ge2$;
\item[(ii)] there exist constants $C_1, C_2>0$ (depending only on $d$) such that for all large $N$,
\[
\frac{C_1}{\nu(\#(S\cap Q_N))}
\le
\frac{\#(S^\ast\cap Q_N)}{\#(S\cap Q_N)}
\le
\frac{C_2}{\nu(\#(S\cap Q_N))};
\]
\item[(iii)] in particular $\dbar(S^\ast\mid S)=0$.
\end{itemize}
\end{lemma}

\begin{remark}
The selection is essentially one--dimensional. We first pass to a $2$--separated
subset $S_{\mathrm{sep}}\subset S$ (in $\|\cdot\|_\infty$); this is used later to
ensure separation of cylinders in the Moran construction (hence diameter and gap
lower bounds). We then enumerate $S_{\mathrm{sep}}$ by nondecreasing $\|\cdot\|_\infty$
and apply the same factorial--block thinning rule as in the one--dimensional case.
The estimates depend only on the counting function $N_{S_{\mathrm{sep}}}(R)$ and
not on the particular enumeration.
\end{remark}

\begin{proof}
Let us define the counting function $N_S(R) = \#\{\mathbf{v} \in S : \|\mathbf{v}\|_\infty \leq R\}$ for $R > 0$, and similarly for other subsets.

We first construct a subset $S_{\text{sep}} \subset S$ that satisfies the separation condition. We define a greedy selection process: order the elements of $S$ by their $\ell_\infty$ norm: $S = \{\mathbf{x}_1, \mathbf{x}_2, \dots\}$.
\begin{itemize}
	\item Select $\mathbf{x}_1$.
	\item Recursively, select the next $\mathbf{x}_j$ only if $\|\mathbf{x}_j - \mathbf{x}_i\|_\infty \ge 2$ for all previously selected $\mathbf{x}_i$.
\end{itemize}
Let the resulting set be $S_{\text{sep}}$. Since an integer point has at most $K_d=3^d$ integer points at $\ell_\infty$-distance strictly less than $2$, at most $K_d$ points from $S$ are discarded for every point selected. Thus, the density of $S_{\text{sep}}$ is comparable to $S$:
\begin{equation} \label{eq:sep_density}
	\frac{1}{K_d} N_S(R) \le N_{S_{\text{sep}}}(R) \le N_S(R).
\end{equation}

We now linearize $S_{\text{sep}}$ by enumerating its elements in increasing order of their $\ell_\infty$ norm:
\[
S_{\text{sep}} = \{\mathbf{u}_1, \mathbf{u}_2, \mathbf{u}_3, \dots\}.
\]
Define the index set
\[
Q = \bigcup_{k=2}^\infty \{ k! + ik : i = 0, 1, \dots, k!-1 \}
\]
and construct the final set $S^\ast$ by selecting vectors from $S_{\text{sep}}$ with indices in $Q$:
\[
S^\ast = \{ \mathbf{u}_n \in S_{\text{sep}} : n \in Q \}.
\]

\noindent\textbf{Verification of properties:}

(i) Since $S^\ast \subset S_{\text{sep}}$, and $S_{\text{sep}}$ is 2-separated by construction, $S^\ast$ satisfies $\|\mathbf{u} - \mathbf{v}\|_\infty \ge 2$ for any distinct $\mathbf{u}, \mathbf{v} \in S^\ast$.

(ii) Let $\xi = N_{S_{\text{sep}}}(R)$ be the number of elements in $S_{\text{sep}}$ up to norm $R$. The number of elements in $S^\ast$ up to norm $R$ is exactly $\#(Q \cap [1, \xi])$. When $\xi \in [k!, (k+1)!)$, we estimate the size of $Q \cap [1, \xi]$ as follows:

Note that by the construction of $Q$, in each interval $[j!, (j+1)!)$  from $j=2$ to $j=k-1$, the set $Q$ contains exactly $j!$ elements.  We also note,  $Q \cap [k!, (k+1)!)$ consists of  a partial block   containing $\left\lfloor \frac{\xi - k!}{k} \right\rfloor + 1$ elements.
We decompose $[1, \xi]$ into three parts:
\begin{align*}
\#(Q \cap [1, \xi]) &= \#(Q \cap [1, (k-1)!)) + \#(Q \cap [(k-1)!, k!)) + \#(Q \cap [k!, \xi]) \\
&\leq (k-1)!+(k-1)!+\frac{\xi-k!}{k}+1 \\
&\leq \frac{3\xi}{k}+1
\leq \frac{4\xi}{\nu(\xi)},
\end{align*}
where we used $(k-1)!\le \xi/k$, which follows directly from $\xi\ge k!$.

For the lower bound, we have:
\[
\begin{aligned}
\#(Q \cap[1, \xi]) & \geq \sum_{i=2}^{k-1} \frac{\#([i!,(i+1)!) \cap \mathbb{N})}{i}+\frac{\#([k!, \xi] \cap \mathbb{N})}{k}-1 \\
& \geq \sum_{i=2}^{k-1} \frac{\#([i!,(i+1)!) \cap \mathbb{N})}{k}+\frac{\#([k!, \xi] \cap \mathbb{N})}{k}-1 \\
& =\frac{\#([2, \xi] \cap \mathbb{N})}{k}-1 \geq \frac{\xi}{2 \nu(\xi)},
\end{aligned}
\]
where the last inequality holds for sufficiently large $\xi$ since $\#([2, \xi] \cap \mathbb{N}) = \lfloor \xi \rfloor - 1 \geq \xi - 2$ and $\frac{\xi - 2}{k} - 1 \geq \frac{\xi}{2k}$ for $\xi \geq 2k + 4$, which is satisfied for sufficiently large $\xi$ given that $\xi \geq k!$ and $k!$ grows faster than $2k+4$.

Substituting $k = \nu(\xi) = \nu(N_{S_{\text{sep}}}(R))$ yields:
\[
\frac{N_{S_{\text{sep}}}(R)}{2\nu(N_{S_{\text{sep}}}(R))} \le N_{S^\ast}(R) \le \frac{4 N_{S_{\text{sep}}}(R)}{\nu(N_{S_{\text{sep}}}(R))}.
\]
Dividing by $N_S(R)$ and using the bounds from \eqref{eq:sep_density}, we obtain:
\[
\frac{1}{2K_d \cdot \nu(N_{S_{\text{sep}}}(R))} \le \frac{N_{S^\ast}(R)}{N_S(R)} \le \frac{4}{\nu(N_{S_{\text{sep}}}(R))}.
\]
Let $k=\nu(N_S(R))$.  Since
$N_S(R)/K_d\le N_{S_{\text{sep}}}(R)\le N_S(R)$ and
$k!/K_d>(k-1)!$ for all sufficiently large $k$, we have
\[
\nu(N_S(R))-1\le \nu(N_{S_{\text{sep}}}(R))\le\nu(N_S(R)).
\]
It follows that the desired bounds hold, for example, with
$C_1=1/(2K_d)$ and $C_2=8$.

(iii) The upper bound in (ii) implies that $\frac{\#(S^\ast\cap Q_N)}{\#(S\cap Q_N)} \to 0$ as $N \to \infty$, hence $\dbar(S^\ast\mid S)=0$.
\end{proof}

\subsection{Extreme seed sets and their dimension}
\label{subsec:step2}

We begin by isolating the ambient ``no--repetition'' sets in which our seed (root)
constructions will take place.  For $S\subset\N^d$ define
\[
E_{S}^{\mathrm{vec}}
:=
\Bigl\{\mathbf{x}\in((0,1)\setminus\Q)^d:
\mathbf a_n(\mathbf x)\in S\ \text{for all }n\ge1,
\ \text{and }\mathbf a_m(\mathbf x)\neq \mathbf a_n(\mathbf x)\ \text{for all }m\neq n
\Bigr\},
\]
where $\mathbf a_n(\mathbf x):=(a_n(x_1),\dots,a_n(x_d))\in\N^d$ denotes the $n$--th
digit vector of $\mathbf x$.

For parameters $t>L>1$ with $Lt<t^2$, any $\mathbf x\in R_{t,L}(K)$ satisfies
\[
t^n< \|\mathbf a_n(\mathbf x)\|_\infty\le Lt^n \qquad(n\ge1).
\]
Since the intervals $(t^n,Lt^n]$ are pairwise disjoint (equivalently $Lt^n<t^{n+1}$),
we have $\mathbf a_m(\mathbf x)\neq \mathbf a_n(\mathbf x)$ whenever $m\neq n$.
In particular, if $K\subset S$, then
\[
R_{t,L}(K)\subset E_{S}^{\mathrm{vec}}.
\]

We also use the divergence subspace
\[
\mathcal E_{\N^d}
:=
\Bigl\{\mathbf{x}\in((0,1)\setminus\Q)^d:
a_n(x_j)\to\infty\ \text{as }n\to\infty\ \text{for each }1\le j\le d
\Bigr\},
\]
and recall the digit--restricted divergence set
\[
\mathcal E_{S}^{\mathrm{vec}}
=
E_{S}^{\mathrm{vec}}\cap \mathcal E_{\N^d}.
\]

Finally, if $S\subset\mathcal C_K$ is uniformly $K$--balanced, then for any
$\mathbf a_n(\mathbf x)\in S$ we have
$\min_j a_n(x_j)\ge K^{-1}\|\mathbf a_n(\mathbf x)\|_\infty$.
Hence $\|\mathbf a_n(\mathbf x)\|_\infty\to\infty$ implies $\mathbf x\in\mathcal E_{\N^d}$,
so in the balanced setting we may regard $\mathcal E_{S}^{\mathrm{vec}}$ as the
natural ambient space for the seed constructions.

Given the set $S^\ast$ obtained in Lemma~\ref{lem:thin-subset-Zd}, we now construct
Moran-type subsets of $E_{\N^d}^{\mathrm{vec}}$ whose digit vectors lie in $S^\ast$ and
grow along a prescribed scale. These sets will serve as \emph{extreme seed sets},
achieving the maximal Hausdorff dimension permitted by the growth properties of $S$.

\begin{definition}[Seed sets associated with $S$]\label{def:seed_sets}
Fix parameters $t>L>1$, with $t$ sufficiently large.
For an infinite subset $K\subset S $ we define
\[
R_{t,L}(K)
:=
\Bigl\{
\mathbf{x}\in E_{\N^d}^{\mathrm{vec}}:
\mathbf a_n(\mathbf x)\in
K\cap\bigl\{\mathbf v\in\N^d:\ t^n< \|\mathbf v\|_\infty\le L t^n\bigr\}
\ \text{for all }n\ge1
\Bigr\}.
\]
Non-empty sets of this form are called \emph{seed sets}.
\end{definition}

\begin{definition}[Extreme seed sets]\label{def:extreme-seed}
Let $S\subset\N^d$ be infinite and let $R_{t,L}(K)$ be a seed set constructed from
some $K\subset S$.
We say that $R_{t,L}(K)$ is an \emph{extreme seed set associated with $S$} if the
following conditions hold:
\begin{itemize}[leftmargin=2em]
  \item[(A1)] $S$ has polynomial density of exponent $\alpha\ge1$;
  \item[(A2)] $K$ contains no nearest neighbour pairs;
  \item[(A3)] $\dbar(K\mid S)=0$;
  \item[(A4)] $\dim_H R_{t,L}(K)\ge d/(2\alpha)$.

\end{itemize}
\end{definition}

The next proposition is a multidimensional adaptation of the extreme--seed
construction in the one--dimensional setting of \cite{Nakajima2025}.

\begin{proposition}[Existence of extreme seed sets]\label{prop:extreme-seed}
Let $S\subset\N^d$ have polynomial density of exponent $\alpha\ge1$.
Then there exist a subset $S^\ast\subset S$ and parameters $t>L>1$, with $t$
sufficiently large, such that $R_{t,L}(S^\ast)$ is an extreme seed set associated
with $S$. Moreover, $S^\ast$ can be chosen \emph{relatively sparse} in $S$, in the
sense that
\begin{align}\label{eq:Sstar-zero-relative-density}
  \dbar(S^\ast\mid S)=0.
\end{align}

If in addition $S\subset\mathcal C_K$ for some $K\ge1$, then the lower bound in
\textup{(A4)} is sharp and in fact
\begin{align}\label{eq:dim-seed-balanced}
  \dim_H R_{t,L}(S^\ast)=\frac{d}{2\alpha}.
\end{align}

In this balanced case, one may also redefine the seed set by restricting the ambient
space to $\cE_{\N^d}$, that is, by replacing $E_{\N^d}^{\mathrm{vec}}$ with $\cE_{\N^d}$ in the
definition of $R_{t,L}(K)$, without affecting any of the above conclusions.
\end{proposition}

\begin{remark}\label{rem:non-divergent-seeds}
Polynomial counting alone does not force coordinatewise divergence.  For example,
$S=\N\times F$ with finite nonempty $F\subset\N^{d-1}$ has polynomial-density
exponent $d$, but one coordinate remains bounded.  Uniform balance is precisely what
converts the norm growth in the seed construction into coordinatewise escape.
\end{remark}

\begin{proof}[Proof of Proposition~\ref{prop:extreme-seed}]
We construct $S^\ast,t,L$ and then verify the sharpness statement in the balanced case.

\medskip
\noindent\textbf{Step 1: Extraction of a thin subset $S^\ast$.}

Since $S\subset\N^d$ has polynomial density with exponent $\alpha\ge1$, there exist constants $C_1,C_2>0$ and $\beta\ge0$ such that for all sufficiently large $N$,
\[
C_1 \frac{N^{d/\alpha}}{(\log N)^\beta} \le |S\cap Q_N| \le C_2 \frac{N^{d/\alpha}}{(\log N)^\beta}.
\]

By Lemma \ref{lem:thin-subset-Zd}, there exists $S^\ast\subset S$ such that:
\begin{itemize}[leftmargin=2em]
  \item[(i)] $S^\ast$ has no nearest neighbour pairs: $\|\mathbf{v}-\mathbf{w}\|_\infty\ge2$ for all distinct $\mathbf{v},\mathbf{w}\in S^\ast$;
  \item[(ii)] There exist constants $c_1,c_2>0$ such that for all large $N$,
  \[
  \frac{c_1}{\nu(|S\cap Q_N|)} \le \frac{|S^\ast\cap Q_N|}{|S\cap Q_N|} \le \frac{c_2}{\nu(|S\cap Q_N|)};
  \]
  \item[(iii)] In particular, $\dbar(S^\ast\mid S)=0$.
\end{itemize}
Thus (A2) and (A3) hold for $K=S^\ast$.

\medskip
\noindent
\textbf{Step 2: Choice of parameters and growth estimates.}

By Lemma~\ref{lem:ratio-asymptotic}, for any $t>L>1$ there exist constants
$C'_1,C'_2>0$, independent of $t$ and $L$, such that for all sufficiently large $n$,
\[
C'_1 L^{-d/\alpha}
\le
\frac{|S\cap Q_{t^n}|}{|S\cap Q_{L t^n}|}
\le
C'_2 L^{-d/\alpha}.
\]

Choose $L>1$ so that
\[
\frac{c_1 C_2'^{-1} L^{d/\alpha}}{2}-c_2 > C_1^{-1}.
\]
Then choose $t>\max\{L,e\}$ so large that the polynomial-density estimates, the
ratio estimates above, the factorial-block comparison below, and
$\nu(|S\cap Q_t|)\ge2$ all hold for every $n\ge1$; enlarge $t$ once more so that
the lower branching number defined below is at least $2$ for every $n$.

For each $n\in\N$, define the annulus
\[
A_n:=\{\mathbf v\in S^\ast:\ t^n< \|\mathbf v\|_\infty\le L t^n\},
\qquad
r_n:=|A_n|.
\]

Fix $n\ge1$ and choose $k=k(n)$ such that
\[
k!\le |S\cap Q_{t^n}|<(k+1)!.
\]
By the left-hand inequality above, for every $n\ge1$,
\[
|S\cap Q_{L t^n}|
\le (C'_1)^{-1} L^{d/\alpha}\,|S\cap Q_{t^n}|
< (k+2)!,
\]
by the choice of $t$.

By the definition of $\nu$, this implies:
\[
0 \le \nu(|S\cap Q_{L t^n}|) - \nu(|S\cap Q_{t^n}|) \le 1.
\]

By Lemma~\ref{lem:nu_bound}, there exists $\lambda > 1$ such that
\[
\nu(|S\cap Q_{t^n}|) \le \nu(t^{nd}) \le (nd\log t)^\lambda = d^\lambda (n\log t)^\lambda
\]
for all $n \ge 1$.

Since we can choose $\lambda' > \lambda$ such that $d^\lambda (n\log t)^\lambda \le (n\log t)^{\lambda'}$ for all $n$, we may therefore assume without loss of generality that
\[
\nu(|S\cap Q_{t^n}|) \le (n\log t)^\lambda \quad \text{for all } n.
\]

We now estimate $r_n$:
\begin{align*}
r_n &= |S^\ast\cap Q_{L t^n}| - |S^\ast\cap Q_{t^n}| \\
&\ge \frac{c_1|S\cap Q_{L t^n}|}{\nu(|S\cap Q_{L t^n}|)} - \frac{c_2|S\cap Q_{t^n}|}{\nu(|S\cap Q_{t^n}|)} \\
&\ge \frac{c_1 {C'_2}^{-1} L^{d/\alpha}|S\cap Q_{t^n}|}{\nu(|S\cap Q_{t^n}|)+1} - \frac{c_2|S\cap Q_{t^n}|}{\nu(|S\cap Q_{t^n}|)} \\
&= \frac{|S\cap Q_{t^n}|}{\nu(|S\cap Q_{t^n}|)} \left( \frac{c_1 {C'_2}^{-1} L^{d/\alpha}}{1 + \frac{1}{\nu(|S\cap Q_{t^n}|)}} - c_2 \right) \\
&\ge \frac{|S\cap Q_{t^n}|}{\nu(|S\cap Q_{t^n}|)} \left( \frac{c_1 {C'_2}^{-1} L^{d/\alpha}}{2} - c_2 \right).
\end{align*}
where the last inequality uses $\nu(|S\cap Q_{t^n}|) \ge 2$.

By our choice of $L$, we have $\frac{c_1 {C'_2}^{-1} L^{d/\alpha}}{2} - c_2 \ge {C_1}^{-1}$, so
\[
r_n \ge \frac{|S\cap Q_{t^n}|}{{C_1}\nu(|S\cap Q_{t^n}|)}.
\]

\medskip
\noindent\textbf{Step 3: Construction of the Moran fractal.}

Using the bounds $|S\cap Q_{t^n}| \ge C_1 t^{n d/\alpha}/(n\log t)^\beta$ and $\nu(|S\cap Q_{t^n}|) \le (n\log t)^\lambda$, we obtain
\begin{equation}\label{eq:rn-estimate}
r_n \ge  \frac{t^{n d/\alpha}}{(n\log t)^\beta} \cdot \frac{1}{(n\log t)^\lambda} =   \frac{t^{n d/\alpha}}{(n\log t)^\rho},
\end{equation}
where $\rho = \beta + \lambda$.

Consider the symbolic space:
\[
\Sigma^\ast = \{(\mathbf{v}_1,\mathbf{v}_2,\ldots) : \mathbf{v}_n \in A_n \text{ for all } n\}
\]
and
\[
\Sigma^\ast_n = \{(\mathbf{v}_1,\mathbf{v}_2,\ldots, \mathbf{v}_n) : \mathbf{v}_i \in A_i \text{ for all } i \in \{1,\ldots,n\}\}.
\]
For each $\omega = (\mathbf{v}_1,\ldots,\mathbf{v}_n) \in \Sigma^\ast_n$, define the cylinder set:
\[
I_\omega = \{\mathbf{x}\in (0,1)^d :\mathbf{a}_i(\mathbf{x})=\mathbf{v}_i \text{ for } i=1,\ldots,n\}.
\]

Define
\[
F:=\bigcap_{n=1}^\infty \ \bigcup_{\omega\in\Sigma_n^\ast}\overline{I_\omega}.
\]
By the $2$--separation of $S^\ast$, sibling cylinders have disjoint closures, and their mesh tends to zero by
Remark~\ref{rmk:uniform-cylinder-decay-d}.  Thus every $\mathbf x\in F$ determines a
unique nested word $(\mathbf v_n)_{n\ge1}$ with $\mathbf v_n\in A_n$.  The associated
infinite continued-fraction expansions are irrational in every coordinate and code
$\mathbf x$, so $F\subset R_{t,L}(S^\ast)$.

We next verify that $F$ is a Moran fractal.
For each $n\ge1$ set $\mathcal Q_n:=\{\overline{I_\omega}:\omega\in\Sigma_n^\ast\}$.
By Remark~\ref{rmk:uniform-cylinder-decay-d},
\[
\max_{Q\in\mathcal Q_n}\diam(Q)
=
\max_{\omega\in\Sigma_n^\ast}\diam(\overline{I_\omega})
\le \sqrt d\,C\,\vartheta^n\to0,
\]
so the vanishing mesh condition holds.

Since $t>L$, the annuli $\{v\in\N^d: t^n< \|v\|_\infty\le Lt^n\}$ are pairwise disjoint.
For each $n\ge1$ set
\[
  \rho_n:=\left\lfloor\frac{t^{nd/\alpha}}{(n\log t)^\rho}\right\rfloor\ge2.
\]
At level $1$, the root $[0,1]^d$ contains $r_1=|A_1|\ge\rho_1$ children.
For $n\ge2$, every parent cylinder $\overline{I_\omega}$ with
$\omega\in\Sigma_{n-1}^\ast$ has $r_n=|A_n|\ge\rho_n$ children.

To obtain separation inside a parent, fix $n\ge1$ and let
$\omega,\omega'\in\Sigma_n^\ast$ be distinct words with the same prefix of length
$n-1$ (the empty prefix when $n=1$). Since distinct vectors in $S^\ast$ satisfy
$\|\mathbf v-\mathbf w\|_\infty\ge2$, Lemma~\ref{lem:sibling_cylinder_separation} yields
\[
\dist(I_\omega,I_{\omega'})\ge
c\,\min_{1\le j\le d}\min\bigl(|I_\omega^{(j)}|,\ |I_{\omega'}^{(j)}|\bigr).
\]
For any $\omega\in\Sigma_n^\ast$, Lemma~\ref{lem:1d_geometry}(i) gives
\[
|I_\omega^{(j)}|\ge \frac12\prod_{i=1}^n (a_i^{(j)}+1)^{-2}.
\]
Using $a_i^{(j)}\le \|\mathbf v_i\|_\infty\le Lt^i$ and $Lt^i+1\le 2Lt^i$,
we obtain
\[
|I_\omega^{(j)}|
\ge 2^{-2n-1}L^{-2n}t^{-n(n+1)}.
\]
Hence, for a constant $c'>0$ depending only on $d$,
\[
\dist(I_\omega,I_{\omega'})\ge \delta_n,
\qquad
\delta_n:=c'(2L)^{-2n}t^{-n(n+1)}.
\]
Therefore $F$ is a Moran fractal with branching numbers $(\rho_n)_{n\ge1}$ and
gap scales $(\delta_n)_{n\ge1}$.

\medskip
\noindent\textbf{Step 4: Hausdorff dimension estimate.}

By Lemma~\ref{lem:moran-lb-nd},
\[
\dim_H F \ge \liminf_{n\to\infty}
\frac{\log\bigl(\prod_{i=1}^{n-1} \rho_i\bigr)}{-\log\bigl(\rho_n^{1/d}\delta_n\bigr)}.
\]
Since $\log \rho_i=\frac{d}{\alpha}i\log t+O(\log i)$, we have
\[
\log\Bigl(\prod_{i=1}^{n-1} \rho_i\Bigr)
=
\sum_{i=1}^{n-1}\log \rho_i
=
\frac{d}{2\alpha}n^2\log t+o(n^2\log t).
\]
Moreover,
\[
-\log\bigl(\rho_n^{1/d}\delta_n\bigr)
=
-\frac1d\log \rho_n-\log\delta_n
=
n(n+1)\log t+o(n^2\log t)
=
n^2\log t+o(n^2\log t),
\]
because $-\log\delta_n=n(n+1)\log t+O(n)$ and $\log \rho_n=O(n)$.
Therefore,
\[
\dim_H F \ge \frac{d}{2\alpha}.
\]
Since $F\subset R_{t,L}(S^\ast)$, it follows that
\[
\dim_H R_{t,L}(S^\ast)\ge \frac{d}{2\alpha},
\]
which establishes \textup{(A4)}. 

This proves the existence statement.  If $S\subset\mathcal C_K$, then norm escape
implies coordinatewise escape, so
$R_{t,L}(S^\ast)\subset\mathcal E_S^{\mathrm{vec}}\subset\mathcal E_S$.
Corollary~\ref{cor:growth-balanced} gives the matching upper bound
$d/(2\alpha)$, proving \eqref{eq:dim-seed-balanced}.  The same observation shows that
restricting the ambient space in the definition of the seed set to
$\mathcal E_{\N^d}$ does not change the set or any conclusion.
\end{proof}

\subsection{Adding digits and elimination maps}
\label{subsec:step3}

Starting from an extreme seed set $R_{t,L}(B)$, we now \emph{insert} finitely many
digit vectors from a target set $A$ along a sparse sequence of levels, so that the
resulting digit set realizes the desired (relative) density while the Hausdorff
dimension is essentially unchanged.  The key technical point is that the associated
\emph{elimination map}, which deletes the inserted blocks, is Hölder with every
exponent $\gamma\in(0,1)$ once suitable growth conditions are enforced.

\subsubsection{Rigorous digit insertion in higher dimensions}
Fix parameters $t>L>1$.
Let $A,B\subset\N^d$ be infinite sets such that $R_{t,L}(B)\neq\emptyset$.
Let $(M_k)_{k\ge1}$ be a strictly increasing sequence with $M_0=0$, and let
$(W_k)_{k\ge1}$ be finite subsets of $A$.
For each $k$, fix an ordering
\[
  W_k=\{w_{k,1},\dots,w_{k,m_k}\},\qquad m_k:=|W_k|.
\]

Given $\by\in R_{t,L}(B)$ with digit sequence
\[
  \bigl(\ba_1(\by),\ba_2(\by),\dots\bigr)\in(\N^d)^{\N},
\]
we define $\bx=\iota(\by)$ by inserting the word $(w_{k,1},\dots,w_{k,m_k})$
\emph{immediately after} $\ba_{M_k}(\by)$ for each $k$, i.e.
\[
  (\ba_n(\bx))_{n\ge1}
  =
  \ba_1(\by),\dots,\ba_{M_1}(\by),
  w_{1,1},\dots,w_{1,m_1},
  \ba_{M_1+1}(\by),\dots,\ba_{M_2}(\by),
  w_{2,1},\dots,w_{2,m_2},\ \dots\ .
\]
We set
\[
  R_{t,L}(B,A,(M_k),(W_k))
  :=
  \{\iota(\by):\by\in R_{t,L}(B)\}.
\]
The \emph{elimination map} $f$ is defined on this set by deleting the inserted
blocks and recovering the underlying seed sequence:
\[
  f(\iota(\by)):=\by.
\]
Thus $f$ is a bijection from $R_{t,L}(B,A,(M_k),(W_k))$ onto $R_{t,L}(B)$.
Writing
\[
P_k:=M_k+\sum_{j<k}m_j,
\]
the $k$th inserted block occupies the absolute positions
$P_k+1,\dots,P_k+m_k$ in every inserted sequence.

\begin{remark}
In our applications we choose the inserted blocks so that no digit vector is repeated
(e.g.\ by ensuring $B\cap \bigcup_k W_k=\emptyset$ and that the $W_k$ are pairwise
disjoint).  This guarantees that the inserted set remains inside the natural
no--repetition ambient space.  The ordering of each $W_k$ is irrelevant for all
dimension estimates and only serves to make $\iota$ (hence $f$) well-defined.
\end{remark}

\subsubsection{Growth conditions}

Let $(\varepsilon_k)_{k\ge1}$ be a decreasing sequence of positive numbers with
$\varepsilon_k\downarrow0$.  We record the quantitative constraints we will impose
(inductively) on $M_k$ (and later on $W_k$).  These conditions are chosen to control
how much the insertion can distort continued-fraction cylinders; compare
\cite[Section~3.4]{Nakajima2025}.

\begin{align}
  Lt^{M_{k-1}}+1 &< M_k^{1/(2d)}\label{growth0}\\
  Lt^{M_{k-1}} + M_k^{1/(2d)} &< t^{M_k}\label{growth1}\\
  \frac{1}{(Lt^{M_{k-1}} + M_k^{1/(2d)} + 1)^{2^{d+1}M_k^{1/2}}}
  &\ge
  \left(\prod_{i=M_{k-1}+1}^{M_k} t^{-2i}\right)^{3\varepsilon_k}\label{growth2}\\
  \inf_{n \ge M_k}
  \frac{\prod_{i=1}^{n+1} (t^{i+2})^{-2}\cdot\left(\prod_{i=1}^n t^{-2i}\right)^{3\varepsilon_k}}
       {\left(\prod_{i=1}^n t^{-2i}\right)^{1+4\varepsilon_k}}
  &\ge 2.\label{growth3}
\end{align}

If the blocks $(W_k)$ have already been specified, we also require
\begin{equation}\label{growth4}
  \prod_{w\in W_k}(\|w\|_\infty+1)^{-2}
  \;\ge\;
  \left(\prod_{i=M_{k-1}+1}^{M_k} t^{-2i}\right)^{3\varepsilon_k}
  \qquad\text{whenever } W_k\neq\emptyset.
\end{equation}

When the seed alphabet is contained in $\mathcal C_K$, we additionally impose
\begin{equation}\label{growth5}
  K^{2n}\le
  \left(\prod_{i=1}^{n}t^{-2i}\right)^{-\varepsilon_k}
  \qquad(n\ge M_k).
\end{equation}

\begin{remark}
Conditions \eqref{growth0}--\eqref{growth4} make the inserted contraction negligible
relative to the seed contraction, while \eqref{growth5} absorbs the bounded-anisotropy
factor.  Together they yield Hölder continuity with every exponent below $1$.
\end{remark}

\begin{lemma}\label{lem:parameter_selection}
Fix $t>L>1$, $K\ge1$, and a sequence $\varepsilon_k\downarrow0$.
Given $M_{k-1}$, one can choose $M_k$ sufficiently large so that
\eqref{growth0}--\eqref{growth3} and \eqref{growth5} hold. In particular, there exists a strictly
increasing sequence $(M_k)_{k\ge1}$ satisfying these conditions.
\end{lemma}

\begin{proof}
Conditions \eqref{growth0}--\eqref{growth1} are immediate. We verify \eqref{growth2} and \eqref{growth3}.

For \eqref{growth2}, the right-hand side equals
\[
   t^{-3\varepsilon_k (M_k+M_{k-1}+1)(M_k-M_{k-1}) } .
\]
The left-hand side is
\[
   t^{-\,2^{d+1}M_k^{1/2}\,\log_t(L t^{M_{k-1}} + M_k^{1/(2d)} + 1)} .
\]
Since
\[
\log_t(L t^{M_{k-1}} + M_k^{1/(2d)} + 1) < M_k^{1/2}
\]
for large \(M_k\), the left-hand side is \(> t^{-2^{d+1}M_k}\). Thus it suffices to check
\[
   t^{-2^{d+1}M_k}
   \ge
   t^{-3\varepsilon_k (M_k+M_{k-1}+1)(M_k-M_{k-1})},
\]
which holds for large \(M_k\).

For \eqref{growth3}, the expression simplifies to
\[
   t^{(n+1)(n\varepsilon_k - 6)}.
\]
After increasing $M_k$, the exponent is increasing for $n\ge M_k$ and satisfies
$(M_k+1)(M_k\varepsilon_k-6)\ge\log_t2$; hence the displayed quantity is at least
$2$ for every $n\ge M_k$.
Finally, since
\[
\left(\prod_{i=1}^{n}t^{-2i}\right)^{-\varepsilon_k}
=t^{\varepsilon_k n(n+1)},
\]
condition \eqref{growth5} also holds for every $n\ge M_k$ after increasing $M_k$.
\end{proof}

\subsubsection{Almost Lipschitz property with explicit estimates}
\begin{theorem}\label{thm:elim-almost-lip}
Fix parameters $t > L > 1$ with $t^2\ge L+1$, and let $K\ge1$.
Let $A,B \subset \N^{d}$ be infinite sets such that $B\subset\mathcal C_K$,
$R_{t,L}(B)$ is non-empty, and
\(\|\mathbf{v}-\mathbf{w}\|_\infty \ge 2\) for all distinct \(\mathbf{v},\mathbf{w}\in B\).
Let $(M_k)_{k\ge 1}$ be a strictly increasing sequence with $M_0=0$,
$(W_k)_{k\ge 1}$ finite subsets of $A$, and
$(\varepsilon_k)_{k\ge 1}$ a decreasing sequence of positive numbers converging to $0$,
all satisfying \eqref{growth3}--\eqref{growth5}.

Then, for every $\gamma\in(0,1)$ there exists $C_\gamma>0$ such that
\begin{equation}\label{eq:holder-gamma}
  \|f(\mathbf{x}_1)-f(\mathbf{x}_2)\|_\infty
  \;\le\; C_\gamma\,\|\mathbf{x}_1-\mathbf{x}_2\|_\infty^{\gamma}
  \quad
  \text{for all }\mathbf{x}_1,\mathbf{x}_2\in
  R_{t,L}(B,A,(M_k),(W_k)).
\end{equation}
Thus $f$ is Hölder continuous with every exponent $\gamma\in(0,1)$ on this set.
\end{theorem}

  \begin{remark}
The hypotheses \eqref{growth3}--\eqref{growth5} are the natural higher--dimensional
counterparts of the one--dimensional assumptions in \cite[Proposition~3.5]{Nakajima2025}.
The proof follows the same strategy, with the only changes coming from the geometry
of $d$--dimensional cylinders (diameter and separation estimates).
\end{remark}

\begin{proof}[Proof of Theorem~\ref{thm:elim-almost-lip}]
Let $\bx_1,\bx_2\in R_{t,L}(B,A,(M_k),(W_k))$ be distinct, set
$\by_i:=f(\bx_i)\in R_{t,L}(B)$, and define
\[
n:=\min\{m\ge0:\ \ba_{m+1}(\by_1)\neq \ba_{m+1}(\by_2)\}.
\]
Since $B$ has no nearest--neighbour pairs, equivalently
\[
n=\min\{m\ge0:\ \|\ba_{m+1}(\by_1)-\ba_{m+1}(\by_2)\|_\infty\ge2\}.
\]
Fix $k\in\N$ and distinguish whether $n<M_k$ or $n\ge M_k$.

\smallskip
\noindent\textbf{Case~1: $n<M_k$.}
Then only finitely many seed prefixes of length below $M_k$ and the fixed inserted blocks preceding them are involved, so $(\bx_1,\bx_2)$ ranges over finitely many cylinder pairs. By Proposition~\ref{prop:cylinder_estimate} and Lemma~\ref{lem:cylinder_digit_separation} there exists $\delta_0=\delta_0(k)>0$ such that $\|\bx_1-\bx_2\|_\infty\ge\delta_0$. Since $\|\by_1-\by_2\|_\infty\le1$, taking $C_\gamma^{(1)}:=\delta_0^{-\gamma}$ yields
\begin{equation}\label{eq:case1-holder-clean}
\|\by_1-\by_2\|_\infty\le C_\gamma^{(1)}\|\bx_1-\bx_2\|_\infty^\gamma.
\end{equation}

\smallskip
\noindent\textbf{Case~2: $n\ge M_k$.}
Let $q\ge k$ be such that $M_q\le n<M_{q+1}$. By construction,
\[
\ba_i(\bx_1)=\ba_i(\bx_2)\qquad \text{for all }1\le i\le n+\sum_{j=1}^q|W_j|.
\]
Hence the first discrepancy for $(\bx_1,\bx_2)$ occurs at
\[
n':=n+\sum_{j=1}^q|W_j|+1.
\]
Moreover, insertions only shift indices, so $\ba_{n'}(\bx_i)=\ba_{n+1}(\by_i)$ for $i=1,2$.
Since distinct vectors in $B$ satisfy $\|\bv-\mathbf{w}\|_\infty\ge 2$, we have
\[
\ba_i(\bx_1)=\ba_i(\bx_2)\ \ (1\le i\le n'-1),
\qquad
\|\ba_{n'}(\bx_1)-\ba_{n'}(\bx_2)\|_\infty
=\|\ba_{n+1}(\by_1)-\ba_{n+1}(\by_2)\|_\infty\ge 2.
\]
Apply Lemma~\ref{lem:cylinder_digit_separation},
\[
\|\bx_1-\bx_2\|_\infty
\ge
C_1\max_{j\in\cD}
\left(
|I(a_1^{(j)}(\bx_1),\ldots,a_{n'-1}^{(j)}(\bx_1))|
\cdot
\frac{|a_{n'}^{(j)}(\bx_1)-a_{n'}^{(j)}(\bx_2)|}
{a_{n'}^{(j)}(\bx_1)\,a_{n'}^{(j)}(\bx_2)}
\right),
\]
where $\cD:=\{1\le j\le d:\ |a_{n'}^{(j)}(\bx_1)-a_{n'}^{(j)}(\bx_2)|\ge 2\}$.
Since $a_{n'}(\bx_i)=a_{n+1}(\by_i)\in B$ and $\by_i\in R_{t,L}(B)$, we have
$a_{n'}^{(j)}(\bx_i)\le Lt^{\,n+1}$ for all $j$ and $i=1,2$, and hence
\[
\max_{j\in\cD}
\frac{|a_{n'}^{(j)}(\bx_1)-a_{n'}^{(j)}(\bx_2)|}
{a_{n'}^{(j)}(\bx_1)\,a_{n'}^{(j)}(\bx_2)}
=
\max_{j\in\cD}
\frac{|a_{n+1}^{(j)}(\by_1)-a_{n+1}^{(j)}(\by_2)|}
{a_{n+1}^{(j)}(\by_1)\,a_{n+1}^{(j)}(\by_2)}
\ge (Lt^{\,n+1})^{-2}.
\]
For every $j\in\cD$, Lemma~\ref{lem:1d_geometry}(i) and the description of the
common prefix give
\[
\begin{aligned}
&|I(a_1^{(j)}(\bx_1),\ldots,a_{n'-1}^{(j)}(\bx_1))|\\
&\qquad\ge
\frac12
\prod_{i=1}^{n}\bigl(a_i^{(j)}(\by_1)+1\bigr)^{-2}
\prod_{\ell=1}^q\prod_{w\in W_\ell}\bigl(w^{(j)}+1\bigr)^{-2}\\
&\qquad\ge
\frac12
\prod_{i=1}^{n}\bigl(\|\ba_i(\by_1)\|_\infty+1\bigr)^{-2}
\prod_{\ell=1}^q\prod_{w\in W_\ell}\bigl(\|w\|_\infty+1\bigr)^{-2}.
\end{aligned}
\]
Combining this uniform coordinatewise lower bound with the preceding estimate for the
next digit, and absorbing the factor $1/2$ into the constant, gives
\[
  \|\bx_1-\bx_2\|_\infty
  \ge
  C_2 (Lt^{n+1})^{-2}
    \prod_{i=1}^{n}\bigl(\|\ba_i(\by_1)\|_\infty+1\bigr)^{-2}
    \prod_{\ell=1}^q\prod_{w\in W_\ell}\bigl(\|w\|_\infty+1\bigr)^{-2}.
\]
Since $t$ is chosen sufficiently large, we may assume $t^2\ge L+1$, hence
$Lt^{i}+1\le (L+1)t^i\le t^{i+2}$ for all $i\ge1$. Therefore,
\[
  \|\bx_1-\bx_2\|_\infty
  \ge
  C_2
  \prod_{i=1}^{n+1} t^{-2(i+2)}
  \prod_{j=1}^q\prod_{w\in W_j}(\|w\|_\infty+1)^{-2}.
\]
By \eqref{growth4}, for each $1\le j\le q$ we have
\[
  \prod_{w\in W_j}(\|w\|_\infty+1)^{-2}
  \ge
  \left(\prod_{s=M_{j-1}+1}^{M_j} t^{-2s}\right)^{3\varepsilon_j}.
\]
Hence
\[
  \|\bx_1-\bx_2\|_\infty
  \ge
  C_2
  \prod_{i=1}^{n+1} t^{-2(i+2)}
  \prod_{j=1}^q
  \left(\prod_{s=M_{j-1}+1}^{M_j} t^{-2s}\right)^{3\varepsilon_j}.
\]

Because $(\varepsilon_j)$ is decreasing and $q\ge k$, we obtain
\[
  \|\bx_1-\bx_2\|_\infty
  \ge
  C_2
  \prod_{i=1}^{n+1} t^{-2(i+2)}
  \left(\prod_{j=1}^k\left(\prod_{i=M_{j-1}+1}^{M_j} t^{-2i}\right)^{3\varepsilon_j}\right)
  \left(\prod_{i=1}^{M_q} t^{-2i}\right)^{3\varepsilon_k}.
\]
Let
\[
  D_k
  :=
  \prod_{j=1}^k\left(\prod_{i=M_{j-1}+1}^{M_j} t^{-2i}\right)^{3\varepsilon_j}.
\]
Since $M_q\le n$, we further have
\[
  \|\bx_1-\bx_2\|_\infty
  \ge
  C_2 D_k
  \prod_{i=1}^{n+1} t^{-2(i+2)}
  \left(\prod_{s=1}^{n} t^{-2s}\right)^{3\varepsilon_k}.
\]
Using \eqref{growth3},
\[
  \inf_{n\ge M_k}
  \frac{
    \prod_{i=1}^{n+1}t^{-2(i+2)}\cdot\left(\prod_{i=1}^n t^{-2i}\right)^{3\varepsilon_k}
  }{
    \left(\prod_{i=1}^n t^{-2i}\right)^{1+4\varepsilon_k}
  }\ge2,
\]
we conclude
\[
  \|\mathbf{x}_1-\mathbf{x}_2\|_\infty
  \ge
  2C_2D_k\,\left(\prod_{i=1}^n t^{-2i}\right)^{1+4\varepsilon_k}.
\]

On the other hand, $B\subset\mathcal C_K$ and
$\|\mathbf a_i(\mathbf y_1)\|_\infty>t^i$ imply
$a_i^{(j)}(\mathbf y_1)>t^i/K$ in every coordinate. Hence the upper cylinder
diameter bound and \eqref{growth5} give
\[
  \|\mathbf{y}_1-\mathbf{y}_2\|_\infty
  \le C_3K^{2n}\prod_{i=1}^n t^{-2i}
  \le C_3\left(\prod_{i=1}^n t^{-2i}\right)^{1-\varepsilon_k}.
\]
Combining this with the preceding lower bound yields
\begin{equation}\label{eq:holder-q-clean}
  \|\mathbf{y}_1-\mathbf{y}_2\|_\infty
  \le
  C_k\|\mathbf{x}_1-\mathbf{x}_2\|_\infty^{\theta_k},
  \qquad
  \theta_k:=\frac{1-\varepsilon_k}{1+4\varepsilon_k}.
\end{equation}

Fix $\gamma\in(0,1)$.  
Since $\varepsilon_k\to0$, there exists $k_0$ with $\theta_{k_0}\ge\gamma$.
Set $C_4:=C_{k_0}$. If $n\ge M_{k_0}$, then \eqref{eq:holder-q-clean},
applied with $k=k_0$, and the bound $\|\bx_1-\bx_2\|_\infty\le1$ give
\[
  \|\by_1-\by_2\|_\infty
  \le
  C_4\|\bx_1-\bx_2\|_\infty^{\theta_{k_0}}
  \le
  C_4\|\bx_1-\bx_2\|_\infty^\gamma.
\]
If $n<M_{k_0}$, applying \eqref{eq:case1-holder-clean} with $k=k_0$
yields the same inequality.  
With
\[
  C_\gamma:=\max\{C_\gamma^{(1)},C_4\},
\]
we obtain
\[
  \|\mathbf{y}_1-\mathbf{y}_2\|_\infty
  \le
  C_\gamma\,\|\mathbf{x}_1-\mathbf{x}_2\|_\infty^\gamma.
\]
This completes the proof.
\end{proof}

\begin{corollary}
\label{cor:dimension-preservation}
With the above construction:
\[
\dim_H R_{t,L}(B, A, (M_k), (W_k)) \geq \dim_H R_{t,L}(B)
\]
\end{corollary}

\begin{proof}
This follows immediately from Theorem~\ref{thm:elim-almost-lip} and Lemma~\ref{lem:almost-lip}.
\end{proof}

	\section{Proofs of the main theorems}\label{sec:proofs}

\subsection{Proof of Theorem~\ref{thm:ubd-main}}
\label{subsec:ubd-proof}

We now prove the universal upper bound and identify the dimension of the full-alphabet
ambient space.

\begin{proof}[Proof of Theorem~\ref{thm:ubd-main}]
Fix $\varepsilon>0$ and choose $N\ge N_0(\varepsilon)$ as in
Theorem~\ref{thm:full-restriction}, so that
\[
  \dim_H\bigl((E_{\ge N})^d\bigr)\le \frac d2+\varepsilon.
\]
If $\mathbf x\in\mathcal E_S$, then for some $n_0$ all coordinate digits satisfy
$a_n(x_j)\ge N$ whenever $n\ge n_0$.  Hence $\mathcal E_S$ is contained in a
countable union of sets obtained from $(E_{\ge N})^d$ by prescribing finitely many
initial digit vectors.  Lemma~\ref{lem:delete-finite} therefore gives
\[
  \dim_H(\mathcal E_S^{\mathrm{vec}})
  \le \dim_H(\mathcal E_S)
  \le \dim_H\bigl((E_{\ge N})^d\bigr)
  \le \frac d2+\varepsilon.
\]
Letting $\varepsilon\downarrow0$ proves the universal upper bound.

For the reverse inequality in the full-alphabet case, recall the one-dimensional set
$E_1$ from the introduction.  A sequence of pairwise distinct positive integers tends
to infinity, so
\[
  E_1^d\subset\mathcal E_{\N^d}^{\mathrm{vec}}.
\]
Ramharter's theorem gives $\dim_H E_1=1/2$, and repeated application of
Lemma~\ref{lem:marstrand-product} yields
\[
  \dim_H(\mathcal E_{\N^d}^{\mathrm{vec}})
  \ge \dim_H(E_1^d)
  \ge d\,\dim_H(E_1)=\frac d2.
\]
Together with the universal upper bound, this proves both full-alphabet equalities.
\end{proof}

\subsection{Proof of Theorem~\ref{thm:Zdd-relative}}
\label{subsec:proof-relative}

In the relative setting we assume throughout that the ambient set $S\subset\N^d$
has polynomial density exponent $\alpha\ge1$ and is uniformly $K$--balanced, i.e.
$S\subset\mathcal C_K$ for some $K\ge1$. Let $A\subset S$ satisfy
$\dbar(A\mid S)>0$.
Proposition~\ref{prop:extreme-seed} then produces an extreme seed root set
$S^\ast\subset S$ with $\dbar(S^\ast\mid S)=0$ and parameters $t>L>1$ such that
\begin{equation*}
  \dim_H R_{t,L}(S^\ast)=\frac{d}{2\alpha}.
\end{equation*}
To retain the density of $A$ while keeping the seed geometry, we will insert digits
from $A\setminus S^\ast$ in sparse windows determined by a fast increasing sequence
$(M_k)$.

\subsubsection{A density bookkeeping lemma}
The next lemma guarantees that one can choose the insertion windows so that the
union of inserted digits recovers the relative density.

\begin{lemma}[Density recovery]\label{lem:density-recovery}
Let $A_1\subset\N^d$ have polynomial density of exponent $\alpha\ge1$, and let
$A_2\subset A_1$ satisfy $\dbar(A_2\mid A_1)>0$.
Fix $t>L>1$.
Then there exists a strictly increasing sequence $(M_k)_{k\ge1}$ such that, if
for each $k\ge1$ we define
\begin{align}\label{eq:def-Ik-Wk-density-recovery}
  I_k
  &:=
  \bigl\{v\in\N^d:\ Lt^{M_{k-1}}< \|v\|_\infty \le Lt^{M_{k-1}}+M_k^{1/(2d)}\bigr\},\\
  W_k
  &:=A_2\cap I_k,
  \nonumber
\end{align}
then
\begin{align}\label{eq:density-recovery-conclusion}
  \dbar\!\left(\bigcup_{k\ge1}W_k\,\Bigm|\,A_1\right)
  \;=\;
  \dbar(A_2\mid A_1).
\end{align}
Moreover, at stage $k$ the integer $M_k$ may be required to exceed any prescribed
threshold depending only on the preceding choices.
\end{lemma}

\begin{proof}
Set $\delta:=\dbar(A_2\mid A_1)>0$.
By definition of $\dbar(\cdot\mid\cdot)$, there exists an increasing sequence
$(L_j)_{j\ge1}$ such that
\begin{equation*}
  \frac{\#(A_2\cap Q_{L_j})}{\#(A_1\cap Q_{L_j})}\longrightarrow \delta
  \qquad (j\to\infty).
\end{equation*}
We inductively choose a subsequence $(L_{j_k})$ and then define $(M_k)$.
Set $M_0:=0$ and assume $M_0,\dots,M_{k-1}$ have been fixed.
Let
\begin{equation*}
  r_{k,\min}:=\bigl\lfloor Lt^{M_{k-1}}\bigr\rfloor\in\N.
\end{equation*}
Since $\#(A_1\cap Q_n)\to\infty$ and $r_{k,\min}$ is fixed at stage $k$, we may
choose $j_k$ arbitrarily large. Thus we may first impose any prescribed lower bound
on $M_k=(L_{j_k}-r_{k,\min})^{2d}$ and then choose $j_k$ so that
$L_{j_k}>r_{k,\min}+1$ and
\begin{align}\label{eq:choose-jk}
  \left|
  \frac{\#(A_2\cap Q_{L_{j_k}})}{\#(A_1\cap Q_{L_{j_k}})}-\delta
  \right|
  <\frac1k,
  \qquad
  \frac{\#(A_1\cap Q_{r_{k,\min}})}{\#(A_1\cap Q_{L_{j_k}})}<\frac1k.
\end{align}
Define 
\begin{equation*}
  m_k:=L_{j_k}-r_{k,\min}\in\N,
  \qquad
  M_k:=m_k^{2d}.
\end{equation*}
Because $\|v\|_\infty\in\N$, the window in \eqref{eq:def-Ik-Wk-density-recovery}
coincides with the integer shell
\begin{align}\label{eq:Ik-shell}
  I_k
  =
  \bigl\{v\in\N^d:\ r_{k,\min}+1\le\|v\|_\infty\le r_{k,\min}+m_k\bigr\}
  =
  \bigl\{v\in\N^d:\ r_{k,\min}+1\le\|v\|_\infty\le L_{j_k}\bigr\}.
\end{align}

For the lower bound, take $R_k:=L_{j_k}$ in
\eqref{eq:density-recovery-conclusion}. Since
$W_k\subset(\bigcup_{m\ge1}W_m)\cap Q_{R_k}$,
\begin{equation*}
  \frac{\#((\bigcup_{m\ge1}W_m)\cap Q_{R_k})}{\#(A_1\cap Q_{R_k})}
  \ge
  \frac{\#(W_k)}{\#(A_1\cap Q_{R_k})}.
\end{equation*}
Moreover, by \eqref{eq:Ik-shell},
\begin{equation*}
  \#(W_k)
  =
  \#(A_2\cap Q_{R_k})-\#(A_2\cap Q_{r_{k,\min}})
  \ge
  \#(A_2\cap Q_{R_k})-\#(A_1\cap Q_{r_{k,\min}}),
\end{equation*}
hence
\begin{equation*}
  \frac{\#(W_k)}{\#(A_1\cap Q_{R_k})}
  \ge
  \frac{\#(A_2\cap Q_{R_k})}{\#(A_1\cap Q_{R_k})}
  -
  \frac{\#(A_1\cap Q_{r_{k,\min}})}{\#(A_1\cap Q_{R_k})}
  >
  \delta-\frac1k-\frac1k
\end{equation*}
by \eqref{eq:choose-jk}. Taking $\limsup_{k\to\infty}$ gives
$\dbar(\bigcup_{m\ge1}W_m\mid A_1)\ge\delta$.

For the matching upper bound, note $\bigcup_{m\ge1}W_m\subset A_2$, so
$\dbar(\bigcup_{m\ge1}W_m\mid A_1)\le\dbar(A_2\mid A_1)=\delta$.
Thus \eqref{eq:density-recovery-conclusion} holds.
\end{proof}

\subsubsection{Windows and a uniform product estimate}
Fix an increasing sequence $(M_k)$ (to be chosen) and define the insertion windows
\begin{align}\label{def_I_k_rela_new}
  I_k
  :=
  \bigl\{v\in\N^d:\ Lt^{M_{k-1}}< \|v\|_\infty \le Lt^{M_{k-1}}+M_k^{1/(2d)}\bigr\}.
\end{align}
The next lemma is the bookkeeping bound needed in the almost--Lipschitz
elimination step.

\begin{lemma}\label{lem:Wk-lower-bound}
Assume $0\le M_{k-1}<M_k$ and let $I_k$ be given by \eqref{def_I_k_rela_new}.
If \eqref{growth0} and \eqref{growth2} hold, then
\begin{align}\label{eq:Ik-product-lowerbd}
  \prod_{w\in I_k}(\|w\|_\infty+1)^{-2}
  \;\ge\;
  \left(\prod_{i=M_{k-1}+1}^{M_k} t^{-2i}\right)^{3\varepsilon_k}.
\end{align}
\end{lemma}

\begin{proof}
Put $r_k:=\lfloor Lt^{M_{k-1}}\rfloor$.  Since $\|v\|_\infty$ is integral,
\begin{equation*}
  \#I_k
  \le
  \bigl\lfloor Lt^{M_{k-1}}+M_k^{1/(2d)}\bigr\rfloor^d-r_k^d
  \le \bigl(Lt^{M_{k-1}}+M_k^{1/(2d)}+1\bigr)^d.
\end{equation*}
Condition \eqref{growth0} gives
$Lt^{M_{k-1}}+M_k^{1/(2d)}+1<2M_k^{1/(2d)}$, hence
\begin{align}\label{eq:Ik-size-bound}
  \#I_k\le (2M_k^{1/(2d)})^d=2^d M_k^{1/2}.
\end{align}
Also, for every $w\in I_k$ we have
$\|w\|_\infty+1 \le Lt^{M_{k-1}}+M_k^{1/(2d)}+1$, so
\begin{equation*}
  \prod_{w\in I_k}(\|w\|_\infty+1)^{-2}
  \ge
  \bigl(Lt^{M_{k-1}}+M_k^{1/(2d)}+1\bigr)^{-2\#I_k}
  \ge
  \bigl(Lt^{M_{k-1}}+M_k^{1/(2d)}+1\bigr)^{-2^{d+1}M_k^{1/2}}
\end{equation*}
by \eqref{eq:Ik-size-bound}.
Finally, \eqref{growth2} is precisely the hypothesis that the last expression is
bounded below by the right-hand side of \eqref{eq:Ik-product-lowerbd}.
\end{proof}

\subsubsection{Proof of Theorem~\ref{thm:Zdd-relative}.}

\begin{proof}[Proof of Theorem~\ref{thm:Zdd-relative}]
Assume $S\subset\N^d$ has polynomial density exponent $\alpha\ge1$ and is
uniformly $K$--balanced, i.e.\ $S\subset\mathcal C_K$ for some $K\ge1$.
Let $A\subset S$ satisfy $\dbar(A\mid S)>0$.
By Proposition~\ref{prop:extreme-seed}, there exist a subset $S^\ast\subset S$
with $\dbar(S^\ast\mid S)=0$ and parameters $t>L>1$ (with $t$ sufficiently large)
such that
\begin{equation*}
  \dim_H R_{t,L}(S^\ast)=\frac{d}{2\alpha}.
\end{equation*}

Since $\dbar(S^\ast\mid S)=0$ and $A\subset S$, we also have
$\dbar(A\cap S^\ast\mid S)=0$, hence
\begin{align}\label{eq:remove-thin-part}
  \dbar(A\setminus S^\ast\mid S)=\dbar(A\mid S).
\end{align}
Fix a decreasing sequence $(\varepsilon_k)_{k\ge1}$ of positive numbers with
$\varepsilon_k\downarrow0$. We choose a strictly increasing sequence $(M_k)_{k\ge1}$
with $M_0=0$ inductively; for each $k$ define $I_k$ by
\eqref{def_I_k_rela_new} and set
\begin{equation*}
  W_k:=(A\setminus S^\ast)\cap I_k.
\end{equation*}

We choose $(M_k)$ so that:
\begin{itemize}
\item the surgery hypotheses \eqref{growth0}--\eqref{growth3} and \eqref{growth5}
hold (so that the almost--Lipschitz elimination applies), and
\item simultaneously, $(M_k)$ satisfies Lemma~\ref{lem:density-recovery} with
$A_1=S$ and $A_2=A\setminus S^\ast$ (so that the inserted digits recover
$\dbar(A\mid S)$).
\end{itemize}
At stage $k$, this is achieved by taking $M_k$ sufficiently large; in particular,
in the proof of Lemma~\ref{lem:density-recovery} one may increase the chosen
radius $L_{j_k}$ further to meet the finitely many growth requirements at that
step.

With these choices, Lemma~\ref{lem:Wk-lower-bound} ensures \eqref{growth4}, hence all hypotheses
of Theorem~\ref{thm:elim-almost-lip} are satisfied. Set
\[
E_{S,A}:=R_{t,L}(S^\ast,A\setminus S^\ast,(M_k),(W_k)).
\]
The elimination map is Hölder with any exponent $\gamma\in(0,1)$, so
\begin{align}\label{eq:dim-lower}
  \dim_H E_{S,A}
  \ge
  \dim_H R_{t,L}(S^\ast)
  =
  \frac{d}{2\alpha}.
\end{align}

On the other hand, every point of $E_{S,A}$ uses only digits from $S$. The background
alphabet $S^\ast$ is disjoint from the inserted alphabet $A\setminus S^\ast$, and
the windows $I_k$ are pairwise disjoint by \eqref{growth1}; hence no digit vector is
repeated. Uniform balance and the lower endpoints of the windows also imply
coordinatewise divergence. Therefore $E_{S,A}\subset\mathcal E_S^{\mathrm{vec}}$, and
Corollary~\ref{cor:growth-balanced} gives
\begin{align}\label{eq:dim-upper}
  \dim_H E_{S,A}
  \le
  \dim_H\mathcal E_S^{\mathrm{vec}}
  \le
  \frac{d}{2\alpha}.
\end{align}

Combining \eqref{eq:dim-lower} and \eqref{eq:dim-upper} gives equality throughout,
so the constructed set has Hausdorff dimension $d/(2\alpha)$.

Finally, by Lemma~\ref{lem:density-recovery} applied with $A_1=S$ and
$A_2=A\setminus S^\ast$, together with \eqref{eq:remove-thin-part}, the union
$\bigcup_{k\ge1}W_k$ satisfies
\begin{equation}\label{eq:final-density}
  \dbar\!\left(\bigcup_{k\ge1}W_k\,\Bigm|\,S\right)
  =
  \dbar(A\setminus S^\ast\mid S)
  =
  \dbar(A\mid S).
\end{equation}
Write $m_k:=|W_k|$. By the definition of $P_k$ above, the block $W_k$ occurs at
the common absolute positions $P_k+1,\dots,P_k+m_k$ in every point of $E_{S,A}$.
Thus
\[
\bigcup_{k\ge1}W_k\subset
\bigcup_{n\ge1}\ \bigcap_{\mathbf x\in E_{S,A}}\{\mathbf a_n(\mathbf x)\}\cap A.
\]
Together with \eqref{eq:final-density} and the reverse inclusion into $A$, this is
exactly the required uniform density identity. The pointwise identity follows from
the same density sandwich. This completes the proof.
\end{proof}

\subsection{Proof of Theorem~\ref{thm:Zdd-fractal-transference}}
We will need a density--recovery lemma that simultaneously enforces divergence: it
selects insertion blocks $W_k\subset A$ whose union has the same upper density as $A$,
while $\min_j w_j\to\infty$ along the blocks.  This will be applied in the proof of
Theorem~\ref{thm:Zdd-fractal-transference}, Part~(a).

\begin{lemma}[Density recovery with divergence]\label{lem:density-recovery-div}
Let $A\subset\N^d$ satisfy $\dbar(A)=:\delta>0$. Fix $t>L>1$.
Then there exists a strictly increasing sequence $(M_k)_{k\ge1}$ such that, if for
each $k\ge1$ we define
\begin{align}\label{eq:def-Ik-Wk-density-recovery-div}
  I_k
  &:=
  \bigl\{v\in\N^d:\ Lt^{M_{k-1}}< \|v\|_\infty \le Lt^{M_{k-1}}+M_k^{1/(2d)}\bigr\},\\
  T_k
  &:=
  \bigl\{v\in\N^d:\ \min_{1\le j\le d} v_j\ge k\bigr\},
  \nonumber\\
  W_k
  &:=
  A\cap I_k\cap T_k,
  \nonumber
\end{align}
then
\begin{align}\label{eq:density-recovery-conclusion-div}
  \dbar\!\left(\bigcup_{k\ge1}W_k\right)=\dbar(A)=\delta,
\end{align}
and moreover every $w\in W_k$ satisfies $\min_{1\le j\le d} w_j\ge k$, hence
\begin{align}\label{eq:mincoord-div}
  \min_{1\le j\le d} w_j \xrightarrow[k\to\infty]{}\infty
  \qquad\text{uniformly for }w\in W_k.
\end{align}
At stage $k$, $M_k$ may additionally be required to exceed any prescribed threshold
depending only on the preceding choices.
\end{lemma}

\begin{proof}
For $m\ge1$ set $A^{(m)}:=A\cap T_m$.
Since $\N^d\setminus T_m=\bigcup_{i=1}^d\{v\in\N^d:\ v_i\le m-1\}$ and
\[
  \bigl|\{v\in Q_N:\ v_i\le m-1\}\bigr|\le (m-1)N^{d-1},
\]
we have $|(\N^d\setminus T_m)\cap Q_N|/|Q_N|\le d(m-1)/N\to0$ as $N\to\infty$.
Hence $\dbar(\N^d\setminus T_m)=0$ and therefore
\begin{align}\label{eq:thick-preserve-dbar}
  \dbar(A^{(m)})=\dbar(A\cap T_m)=\dbar(A)=\delta
  \qquad\text{for every fixed }m.
\end{align}

We now inductively choose $(M_k)$ together with auxiliary radii $(R_k)$.
Set $M_0:=0$. Assume $M_0,\dots,M_{k-1}$ have been fixed and define
\begin{equation*}
  r_{k,\min}:=\bigl\lfloor Lt^{M_{k-1}}\bigr\rfloor\in\N.
\end{equation*}
Using \eqref{eq:thick-preserve-dbar} with $m=k$, we can choose $R_k\in\N$
arbitrarily large. Thus we may first impose any prescribed lower bound on
$M_k=(R_k-r_{k,\min})^{2d}$ and then choose $R_k>r_{k,\min}+1$ so that
\begin{align}\label{eq:choose-Rk-div}
  \left|
  \frac{|A^{(k)}\cap Q_{R_k}|}{|Q_{R_k}|}-\delta
  \right|<\frac1k,
  \qquad
  \frac{|Q_{r_{k,\min}}|}{|Q_{R_k}|}<\frac1k.
\end{align}
Set
\begin{equation*}
  m_k:=R_k-r_{k,\min}\in\N,
  \qquad
  M_k:=m_k^{2d}.
\end{equation*}
Then $M_k^{1/(2d)}=m_k$, and since $\|v\|_\infty\in\N$ we have the identity of shells
\begin{align}\label{eq:Ik-as-shell-div}
  I_k
  =
  \bigl\{v\in\N^d:\ r_{k,\min}+1\le \|v\|_\infty\le r_{k,\min}+m_k\bigr\}
  =
  \bigl\{v\in\N^d:\ r_{k,\min}+1\le \|v\|_\infty\le R_k\bigr\}.
\end{align}
Define $W_k:=A\cap I_k\cap T_k=A^{(k)}\cap I_k$ as in \eqref{eq:def-Ik-Wk-density-recovery-div}.

For the lower bound on $\dbar(\bigcup_{m\ge1}W_m)$, evaluate at $N=R_k$.
Since $W_k\subset(\bigcup_{m\ge1}W_m)\cap Q_{R_k}$, we get
\begin{equation*}
  \frac{|(\bigcup_{m\ge1}W_m)\cap Q_{R_k}|}{|Q_{R_k}|}
  \ge
  \frac{|W_k|}{|Q_{R_k}|}.
\end{equation*}
Moreover, by \eqref{eq:Ik-as-shell-div},
\begin{equation*}
  |W_k|
  =
  |A^{(k)}\cap Q_{R_k}|-|A^{(k)}\cap Q_{r_{k,\min}}|
  \ge
  |A^{(k)}\cap Q_{R_k}|-|Q_{r_{k,\min}}|.
\end{equation*}
Dividing by $|Q_{R_k}|$ and using \eqref{eq:choose-Rk-div} yields
\begin{equation*}
  \frac{|W_k|}{|Q_{R_k}|}
  \ge
  \frac{|A^{(k)}\cap Q_{R_k}|}{|Q_{R_k}|}-\frac{|Q_{r_{k,\min}}|}{|Q_{R_k}|}
  >
  \delta-\frac1k-\frac1k.
\end{equation*}
Taking $\limsup_{k\to\infty}$ gives
$\dbar(\bigcup_{m\ge1}W_m)\ge\delta$.

For the reverse inequality, note $\bigcup_{m\ge1}W_m\subset A$, hence
$\dbar(\bigcup_{m\ge1}W_m)\le\dbar(A)=\delta$. This proves
\eqref{eq:density-recovery-conclusion-div}. Finally, \eqref{eq:mincoord-div} follows
immediately from $W_k\subset T_k$.
\end{proof}

\subsubsection{Proof of Theorem~\ref{thm:Zdd-fractal-transference}, Part~(a)}
 \begin{proof}[Proof of Theorem~\ref{thm:Zdd-fractal-transference}, Part~(a)]
Fix $S\subset\N^d$ with $\dbar(S)=:\delta>0$.
We construct a set $E_S\subset\mathcal E_{\N^d}^{\mathrm{vec}}$
such that
\begin{equation*}
  \dbar\bigl(D(\bx)\cap S\bigr)=\delta
  \qquad\text{for every }\bx\in E_S,
\end{equation*}
and $\dim_H(E_S)=d/2$.

Fix $K_0\ge2$ and set $S_0:=\mathcal C_{K_0}$.
Since $|S_0\cap Q_N|\asymp N^d$, the set $S_0$ has polynomial density exponent $1$.
Applying Proposition~\ref{prop:extreme-seed} to $S_0$ (with $\alpha=1$), we obtain
parameters $t>L>1$ (with $t$ sufficiently large) and a subset $B\subset S_0$ such that
\begin{equation*}
  \mathcal R:=R_{t,L}(B)
\end{equation*}
is non-empty and
\begin{equation*}
  \dim_H(\mathcal R)=\frac d2.
\end{equation*}
Moreover, every $\by\in\mathcal R$ satisfies $\ba_n(\by)\in B\subset\mathcal C_{K_0}$ and
$t^n< \|\ba_n(\by)\|_\infty\le Lt^n$, hence
\begin{align}\label{eq:background-divergence}
  \min_{1\le j\le d} a_n(y_j)
  \ge
  \frac{1}{K_0}\,\|\ba_n(\by)\|_\infty
  \ge
  \frac{t^n}{K_0}
  \longrightarrow\infty.
\end{align}
In particular, $\mathcal R\subset\mathcal E_{\N^d}^{\mathrm{vec}}$.

Since $\dbar(B\mid\mathcal C_{K_0})=0$ and $|\mathcal C_{K_0}\cap Q_N|\asymp N^d$,
we have $\dbar(B)=0$. Put $A:=S\setminus B$, so $\dbar(A)=\dbar(S)=\delta$.
Fix a decreasing sequence $(\varepsilon_k)_{k\ge1}$ of positive numbers with
$\varepsilon_k\downarrow0$. Apply Lemma~\ref{lem:density-recovery-div} to $A$ with
the above parameters $t>L>1$, using its flexibility clause to impose
\eqref{growth0}--\eqref{growth3} and \eqref{growth5} at the same stages. We obtain a strictly increasing sequence
$(M_k)_{k\ge1}$ and finite sets $W_k\subset A$ such that, writing
$W:=\bigcup_{k\ge1}W_k$,
\begin{align}\label{eq:W-density-part-a-final}
  \dbar(W)=\dbar(S)=\delta,
  \qquad
  \min_{1\le j\le d} w_j\ge k\ \ \text{for all }w\in W_k.
\end{align}
Since $W_k\subset I_k$, Lemma~\ref{lem:Wk-lower-bound} gives \eqref{growth4}.

Define
\begin{equation*}
  E_S:=R_{t,L}\bigl(B,A,(M_k),(W_k)\bigr).
\end{equation*}
By construction, the block $W_k$ is inserted (in a fixed order) at the digit positions
$\{P_k+1,\dots,P_k+|W_k|\}$, independently of the choice of $\bx\in E_S$.
Set
\[
D_*:=\bigcup_{n\ge1}\ \bigcap_{\bx\in E_S}\{\ba_n(\bx)\}.
\]
The common positions give $W\subset D_*\cap S\subset S$, and therefore
\begin{equation*}
  \dbar(D_*\cap S)=\dbar(S).
\end{equation*}
In particular, for every $\bx\in E_S$,
\[
\dbar(S)=\dbar(W)\le\dbar(D(\bx)\cap S)\le\dbar(S).
\]
The background alphabet $B$ is disjoint from $A$, and the insertion windows are
pairwise disjoint by \eqref{growth1}; together with the $2$--separation of $B$, this
also gives vector--injectivity.

Let $\bx\in E_S$. Along the background positions, the digit vectors of $\bx$
are taken from $B\subset\mathcal C_{K_0}$ at scales comparable to $t^n$; therefore
\eqref{eq:background-divergence} implies coordinatewise divergence along those indices.
Along the inserted positions belonging to the $k$th inserted block, every digit vector
lies in $W_k$, hence satisfies $\min_j a_n(x_j)\ge k$ by \eqref{eq:W-density-part-a-final}.
Since $k\to\infty$ as we move to later blocks, the inserted digits also diverge
coordinatewise. Hence $\ba_n(\bx)\to\infty$ coordinatewise, and therefore
\begin{align}\label{eq:FS-in-Evec}
  E_S\subset\mathcal E_{\N^d}^{\mathrm{vec}}.
\end{align}

Let $f:E_S\to\mathcal R$ be the elimination map deleting all inserted blocks.
By Theorem~\ref{thm:elim-almost-lip}, $f$ is H\"older for every $\gamma\in(0,1)$ and is
surjective. Hence
\begin{equation*}
  \dim_H(E_S)\ge \dim_H(\mathcal R)=\frac d2.
\end{equation*}
Together with \eqref{eq:FS-in-Evec} and $\dim_H(\mathcal E_{\N^d}^{\mathrm{vec}})=d/2$,
we obtain $\dim_H E_S=d/2$. Since $E_S$ is contained in the pointwise set in
\eqref{eq:ftp-a-pointwise}, that set also has Hausdorff dimension $d/2$. This proves
both \eqref{eq:ftp-a-uniform} and \eqref{eq:ftp-a-pointwise}.
\end{proof}

\subsubsection{Proof of Theorem~\ref{thm:Zdd-fractal-transference}, Part (b)}

We will use Lemma~\ref{lem:density-reduction} to convert Banach density information
into deterministic insertion blocks: it replaces $S$ by $S'\subset S$ with
$\dbar(S')=0$ and $\dB(S')=\dB(S)$ and provides witnesses $Q_{N_k}(\bv_k)$ with
$\min_i(v_k)_i\to\infty$.  Setting $W_k:=S'\cap Q_{N_k}(\bv_k)$ gives dense blocks
whose digits drift to infinity in every coordinate.

Compared with the one--dimensional argument in \cite{Nakajima2025}, this is the
additional multidimensional step: in $\N^d$ coordinatewise divergence is not
automatic (cf.\ $\N\times F$), and Lemma~\ref{lem:density-reduction} enforces it
by pushing the Banach witnesses off to infinity.

\begin{proof}[Proof of Theorem~\ref{thm:Zdd-fractal-transference}, Part (b)]

Write $\widehat S:=S$ and let $\beta:=\dB(\widehat S)>0$. If
$\dB(\widehat S)=\dbar(\widehat S)$, take $F_S:=E_{\widehat S}$ from Part~(a) and set
\[
D_*:=\bigcup_{n\ge1}\ \bigcap_{\mathbf x\in F_S}\{\mathbf a_n(\mathbf x)\}.
\]
Then $\dbar(D_*\cap\widehat S)=\beta$, so
$\dB(D_*\cap\widehat S)\ge\beta$; the reverse inequality follows from
$D_*\cap\widehat S\subset\widehat S$. Thus Part~(b) holds in this case.

Assume henceforth that $\dB(\widehat S)>\dbar(\widehat S)$.
Apply Lemma~\ref{lem:density-reduction} to $\widehat S$ to obtain a subset
$S_0\subset\widehat S$, a strictly
increasing sequence $(N_k)_{k\ge1}$, and vectors $(\mathbf v_k)_{k\ge1}\subset\N^d$ such that
$\dB(S_0)=\dB(S)=\beta$, $\dbar(S_0)=0$, $\min_i (v_k)_i\to\infty$, and
\begin{equation}\label{eq:Banach-witness-b}
  \frac{|S_0\cap Q_{N_k}(\mathbf v_k)|}{|Q_{N_k}|}\longrightarrow \beta.
\end{equation}
Replacing $S$ by $S_0$ (and relabelling), we may assume $\dbar(S)=0$ and
\eqref{eq:Banach-witness-b} holds with $\beta=\dB(S)$.

The boxes $\mathcal B_k:=Q_{N_k}(\mathbf v_k)$ supplied by
Lemma~\ref{lem:density-reduction} are pairwise disjoint. Define $W_k:=S\cap\mathcal B_k$ and
$m_k:=|W_k|$. Then \eqref{eq:Banach-witness-b} becomes
\begin{equation}\label{eq:Wk-density-b}
  \frac{|W_k|}{|Q_{N_k}|}\longrightarrow \dB(S)=\beta.
\end{equation}
Moreover, since $W_k\subset \mathbf v_k+Q_{N_k}$, every $w\in W_k$ satisfies
$\min_i w_i\ge \min_i (v_k)_i+1$, hence $\min_i w_i\to\infty$ uniformly for $w\in W_k$.

Since $\dbar(S)=0$, the complement $S^c:=\N^d\setminus S$ has upper density $1$.
By Lemma~\ref{lem:balanced-density}, there exists $K\ge1$ such that
$\dbar(S^c\cap \mathcal C_K)>0$, where
$\mathcal C_K:=\{\mathbf a\in\N^d:\max_j a_j/\min_j a_j\le K\}$.
Set $A:=S^c\cap \mathcal C_K$. Since $\dbar(S)=0$ and
$|\mathcal C_K\cap Q_N|\sim c_KN^d$ for some $c_K>0$,
\[
|A\cap Q_N|=c_KN^d+o(N^d).
\]
Thus $A$ is uniformly $K$--balanced and has polynomial density exponent $1$.

Applying Proposition~\ref{prop:extreme-seed} to $A$, we obtain parameters $t>L>1$
(with $t$ sufficiently large) and an infinite $2$--separated set $B\subset A$ such that
$\mathcal R:=R_{t,L}(B)$ is non-empty and $\dim_H\mathcal R=d/2$.
Since $B\subset\mathcal C_K$, every $\mathbf y\in\mathcal R$ satisfies
$t^n< \|\mathbf a_n(\mathbf y)\|_\infty\le Lt^n$ and
$\min_j a_n(y_j)\ge \|\mathbf a_n(\mathbf y)\|_\infty/K\to\infty$, so $\mathcal R\subset\mathcal E_{\N^d}^{\mathrm{vec}}$.

Fix a decreasing sequence $(\varepsilon_k)_{k\ge1}$ with $\varepsilon_k\downarrow0$.
We choose a strictly increasing sequence $(M_k)_{k\ge1}$ with $M_0=0$ such that, for every
$k\ge1$, the data
$\bigl(t,L;\, A:=S,\ B;\, (M_i),\ (W_i),\ (\varepsilon_i)\bigr)$ satisfy
\eqref{growth3}--\eqref{growth5}.  This is possible because $W_k$ is fixed at stage
$k$: the left-hand side of \eqref{growth4} is then a fixed positive number, whereas
its right-hand side tends to zero as $M_k\to\infty$; Lemma~\ref{lem:parameter_selection}
provides \eqref{growth3} and \eqref{growth5}.

Define
\begin{equation}\label{eq:def-FS-b}
  F_S:=R_{t,L}\bigl(B,S,(M_k),(W_k)\bigr).
\end{equation}
By construction, every inserted digit lies in $S$, while every background digit lies in
$B\subset S^c$, so the two digit alphabets are disjoint. Since the boxes $\mathcal B_k$ are disjoint,
the sets $W_k$ are disjoint as well, and $B$ is $2$--separated; hence no digit vector is
repeated in the construction.

We next verify divergence. Background digits come from $\mathcal R\subset\mathcal E_{\N^d}^{\mathrm{vec}}$ and
therefore diverge coordinatewise along the background positions. Inserted digits belong to
$W_k$ and satisfy $\min_i w_i\to\infty$ along the inserted blocks. Hence every
$\mathbf x\in F_S$ has coordinatewise digits diverging, and thus $F_S\subset\mathcal E_{\N^d}^{\mathrm{vec}}$.

Let $f:F_S\to\mathcal R$ be the elimination map deleting all inserted blocks.
By Theorem~\ref{thm:elim-almost-lip}, $f$ is H\"older for every $\gamma\in(0,1)$ and surjective.
Hence $\dim_H F_S\ge \dim_H\mathcal R=d/2$. Together with $F_S\subset\mathcal E_{\N^d}^{\mathrm{vec}}$ and
$\dim_H(\mathcal E_{\N^d}^{\mathrm{vec}})=d/2$, we obtain
\begin{align}\label{eq:dim-FS-b}
  \dim_H F_S=\frac d2.
\end{align}

Finally set
\[
  D_*:=\bigcup_{n\ge1}\ \bigcap_{\mathbf x\in F_S}\{\mathbf a_n(\mathbf x)\}.
\]
For each $k$, the construction \eqref{eq:def-FS-b} inserts the block $W_k$ (in a fixed order)
at the common absolute positions $\{P_k+1,\dots,P_k+m_k\}$, independently of
$\mathbf x\in F_S$.
Thus $W_k\subset D_*\cap S$ for every $k$, and therefore
\begin{align*}
  \frac{|(D_*\cap S)\cap \mathcal B_k|}{|Q_{N_k}|}
  \ge
  \frac{|W_k|}{|Q_{N_k}|}.
\end{align*}
Taking $\limsup$ and using \eqref{eq:Wk-density-b} gives
$\dB(D_*\cap S)\ge\beta$, while $D_*\cap S\subset S$ gives the reverse inequality.
Since $S=S_0\subset\widehat S$ and $\dB(\widehat S)=\beta$, we further have
\[
\beta=\dB(D_*\cap S)\le\dB(D_*\cap\widehat S)\le\dB(\widehat S)=\beta.
\]
Thus the required uniform identity holds for the original set $\widehat S$.
The pointwise conclusion follows because $F_S$ is contained in the corresponding
pointwise set and has Hausdorff dimension $d/2$.
\end{proof}

\section*{Acknowledgements}

We are grateful to Professors Yubin He and Leiye Xu for helpful discussions, and to Rongzhong Xiao for valuable comments.
This work is supported by the National Key R\&D Program of China (Nos.~2024YFA1013602 and 2024YFA1013600).

\bigskip
\noindent
Zhuowen Guo\\
\textsc{School of Mathematical Sciences, University of Science and Technology of China} \par\nopagebreak
\noindent
\href{mailto:guozw0920@mail.ustc.edu.cn}
{\texttt{guozw0920@mail.ustc.edu.cn}}

\bigskip
\noindent
Kangbo Ouyang\\
\textsc{School of Mathematical Sciences, University of Science and Technology of China} \par\nopagebreak
\noindent
\href{mailto:oy19981231@mail.ustc.edu.cn}
{\texttt{oy19981231@mail.ustc.edu.cn}}

\bigskip
\noindent
Jiahao Qiu\\
\textsc{School of Mathematical Sciences, University of Science and Technology of China} \par\nopagebreak
\noindent
\href{mailto:qiujh@mail.ustc.edu.cn}
{\texttt{qiujh@mail.ustc.edu.cn}}

\bigskip
\noindent
Shuhao Zhang\\
\textsc{School of Mathematical Sciences, University of Science and Technology of China} \par\nopagebreak
\noindent
\href{mailto:yichen12@mail.ustc.edu.cn}
{\texttt{yichen12@mail.ustc.edu.cn}}

\end{document}